\documentclass[11pt,reqno]{amsart}
\usepackage[T1]{fontenc}
\usepackage{caption}
\usepackage{amsmath}%
\numberwithin{equation}{section}
\usepackage{amsthm}
\usepackage{amsxtra}%
\usepackage{amsfonts}%
\usepackage{amssymb}%
\usepackage[margin=1.25in]{geometry}
\usepackage{hyperref}
\usepackage{graphicx}
\usepackage{subfig}
\usepackage{cite}
\hypersetup{colorlinks,breaklinks,
             linkcolor=blue,urlcolor=blue,
             anchorcolor=blue,citecolor=blue}
\usepackage{xcolor}
\usepackage{array}
\usepackage{booktabs}
\usepackage[foot]{amsaddr}

\newcommand{\e}{\mathbf{e}}
\newcommand{\osc}[1]{\underset{#1}{\,\text{osc}\,}}

\newcommand{\coti}{\text{cot}^{-1}}

\renewcommand{\O}{{\mathcal O}}
\renewcommand{\bar}[1]{{\overline{#1}}}

\renewcommand{\H}{{\mathcal H}}
\newcommand{\R}{\mathbb{R}}

\newcommand{\eps}{\varepsilon}

\renewcommand{\phi}{\varphi}

\renewcommand{\div}{\text{div}}
\newcommand{\diva}{\text{div}\ }
\newcommand{\atan}[2]{\text{atan2}(#1,#2)}

\def\XXint#1#2#3{{\setbox0=\hbox{$#1{#2#3}{\int}$ }
\vcenter{\hbox{$#2#3$ }}\kern-.6\wd0}}

\newcommand{\eg}{e.g., }

\def\hexnumber#1{\ifcase#1 0\or1\or2\or3\or4\or5\or6\or7\or8\or9\or
 A\or B\or C\or D\or E\or F\fi}
\edef\msbhx{\hexnumber\symAMSb}
\mathchardef\emptyset="0\msbhx3F

\def\Bbb#1{{\mathbb#1}}
\def\R{\Bbb R}  \def\Rx{\R\mkern1mu^}

\newtheorem{theorem}{Theorem}
\newtheorem{lemma}[theorem]{Lemma}

\theoremstyle{definition}
\newtheorem{remark}[theorem]{Remark}

\makeatletter
\renewcommand{\email}[2][]{%
  \ifx\emails\@empty\relax\else{\g@addto@macro\emails{,\space}}\fi%
  \@ifnotempty{#1}{\g@addto@macro\emails{\textrm{(#1)}\space}}%
  \g@addto@macro\emails{#2}%
}
\makeatother

\def\dsty{\displaystyle} 

\def\Eq#1$$#2$${\StEq#1  \EnEq{#2}}
\def\StEq#1 #2\EnEq#3{\begin{equation}\label{#1} #3\end{equation}}

\newdimen\eqjot \eqjot = 1\jot
\def\addtab#1={#1\;&=}
\def\addtabe#1=#2={#1=#2\;&=}
\def\openupeq{\openup \the\eqjot}

\def\ezeq#1#2#3{{\def\\{\cr#1}\vcenter{\openupeq \halign{$\displaystyle 
   \hfil##$&$\displaystyle##\hfil$&&\hskip#2pt$\displaystyle##\hfil
        $\cr#1#3\cr}}}}
\def\eaeq{\ezeq\addtab}

\def\eeq{\eaeq{20}}

\def\bo#1{{\bf #1}}
\def\ro#1{{\rm #1}} \def\rbox#1{\hbox{\rm #1}}

\def\rox#1{\quad \rbox{#1}\quad }

\def\roq#1{\qquad \rbox{#1}\qquad }

\def\bpa#1{\bigl(\,{#1}\,\bigr)}
\def\bbk#1{\bigl [\,{#1}\,\bigr ]}

\def\inside#1{\ro{int}\>#1}  \def\intS{\inside S}

\def\capz{\,\cap\,}

\def\C{\textrm{C}} \def\O{\textrm{O}}

\def\v{\bo v}  
\def\V{\bo V} \def\W{\bo W}  

\def\tha{\theta} \def\vph{\varphi}

\def\ostrut#1#2{\hbox{\vrule height #1pt depth #2pt width 0pt}}

\title[Computation of Invariants via Boundary Integrals]{Computation of Circular Area and Spherical Volume Invariants via Boundary Integrals}

\author{Riley O'Neill, Pedro Angulo-Umana, Jeff Calder, Bo Hessburg, Peter J.~Olver, Chehrzad Shakiban, \and Katrina Yezzi-Woodley}
\address[O'Neill, Shakiban]{Department of Mathematics, University of St.~Thomas}
\address[Angulo-Umana, Calder, Hessburg, Olver]{School of Mathematics, University of Minnesota}
\address[Yezzi-Woodley]{Department of Anthropology, University of Minnesota}
\email{riley.oneill@stthomas.edu, angul010@umn.edu, jcalder@umn.edu, hessb017@umn.edu, olver@umn.edu, cshakiban@stthomas.edu, yezz0003@umn.edu}
\thanks{{\bf Funding}: The authors are grateful to the National Science Foundation for support under grant NSF-DMS:1816917, the University of St.~Thomas Center for Applied Math, and a University of Minnesota Grant in Aid award. \hfill\break \ostrut11\quad\ {\bf Source Code}: \href{https://github.com/jwcalder/Spherical-Volume-Invariant}{https://github.com/jwcalder/Spherical-Volume-Invariant}}

\begin{document}

\begin{abstract}
We show how to compute the circular area invariant of planar curves, and the spherical volume invariant of surfaces, in terms of line and surface integrals, respectively. We use the Divergence Theorem to express the area and volume integrals as line and surface integrals, respectively, against particular kernels; our results also extend to higher dimensional hypersurfaces. The resulting surface integrals are computable \emph{analytically} on a triangulated mesh. This gives a simple computational algorithm for computing the spherical volume invariant for triangulated surfaces that does not involve discretizing the ambient space.  We discuss potential applications to feature detection on broken bone fragments of interest in anthropology.
\end{abstract}

\maketitle

\section{Introduction}

The aim of this paper is to facilitate the computation of certain integral invariants that have been proposed for applications in digital image processing, namely, the circular area and spherical volume invariants, as defined below.  We show that both can be efficiently evaluated by reducing them to boundary integrals --- line or surface integrals, respectively, --- plus an additional term that depends only on the local surface geometry, thus enabling them to be computed directly from the curve or surface image data.

More specifically, given a Jordan plane curve $C \subset \Rx2$ 
with interior $\Omega = \ro{int}\>C$,  at each point $p$ in the curve $C$, the value of the (local) \emph{circular area invariant} of radius $r > 0$ at $p$ is defined as the area (Lebesgue measure)  of the region given by the intersection of the interior of the curve with a disk of radius $r$ centered at the point $p$, denoted $D_r(p)$:
\begin{equation}\label{eq:Cai}
A_{C,r}(p) = A(D_r(p) \capz \Omega ).
\end{equation}
The circular area is clearly invariant under Euclidean motions of the curve, of course assuming one relates the base points $p$ accordingly. The ability of the local circular area invariant to uniquely characterize the curve up to Euclidean motion is discussed in detail in \cite{CaEs}. See Figure \ref{fig:CA} for an illustration. 
For sufficiently smooth curves, e.g.~$\C^3$, the circular area invariant is related to the curvature $\kappa (p) $ at the point $p \in C$ by the asymptotic expansion \cite{CaEs}
\begin{equation}\label{eq:Caiasym}
A_{C,r}(p) = \frac{\pi r^2}{2}- \frac{1}{3}\kappa(p) r^3 + \O(r^4) \ \ \text{ as }r \to 0.
\end{equation}
A global invariant can be obtained by averaging over the curve:
\Eq{gainv}
$$\widetilde A_{C,r} = \frac 1{L}\oint_{C}A_{C,r}(p(s))\,ds,$$
where length $l(C)$ denotes the length of $C$.

\begin{figure}
\centering
\subfloat[Circular Area Invariant]{\includegraphics[width=.4\textwidth]{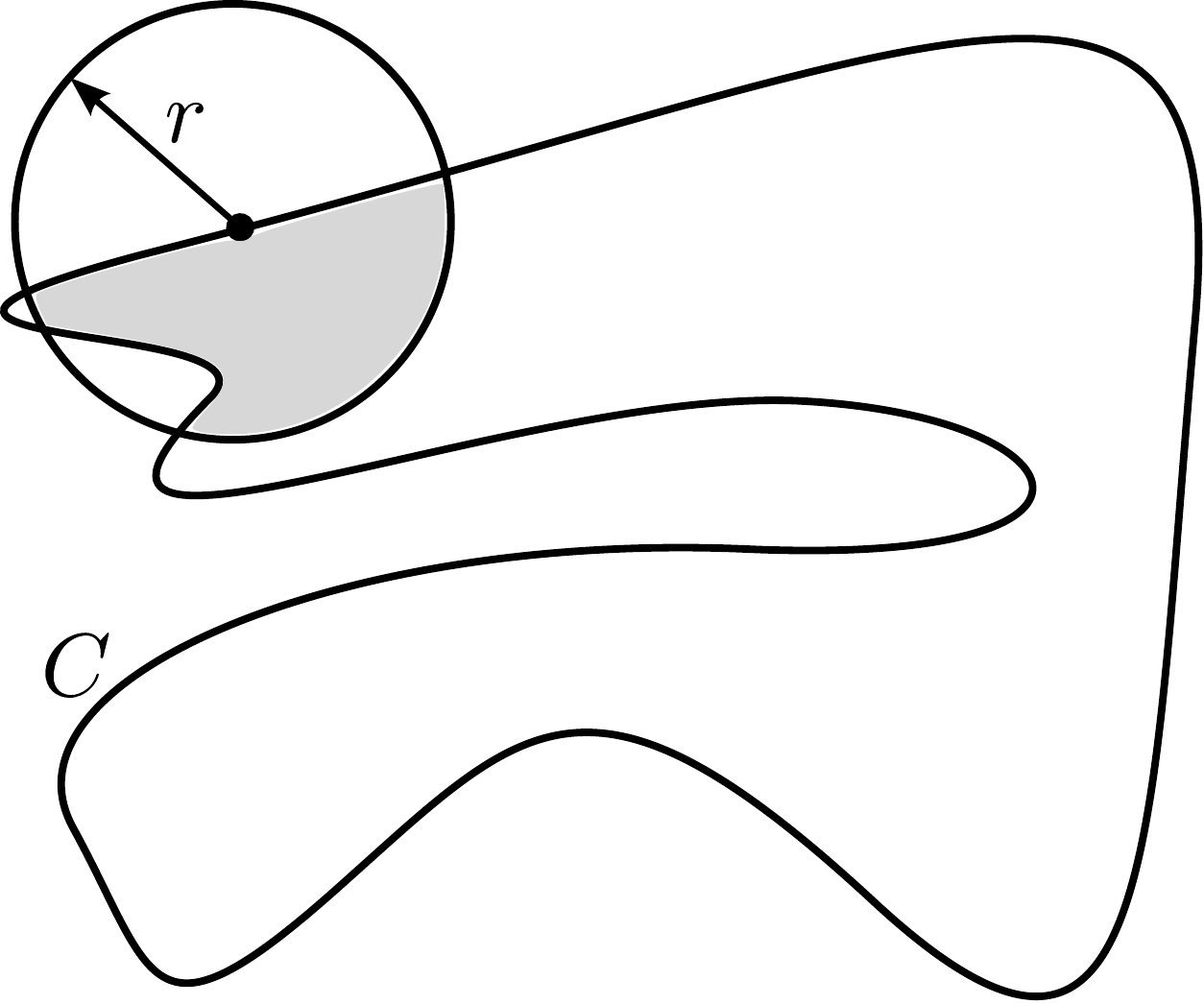}\label{fig:CA}}
\subfloat[Spherical Volume Invariant]{\includegraphics[width=0.45\textwidth,clip=true,trim = 150 130 150 80]{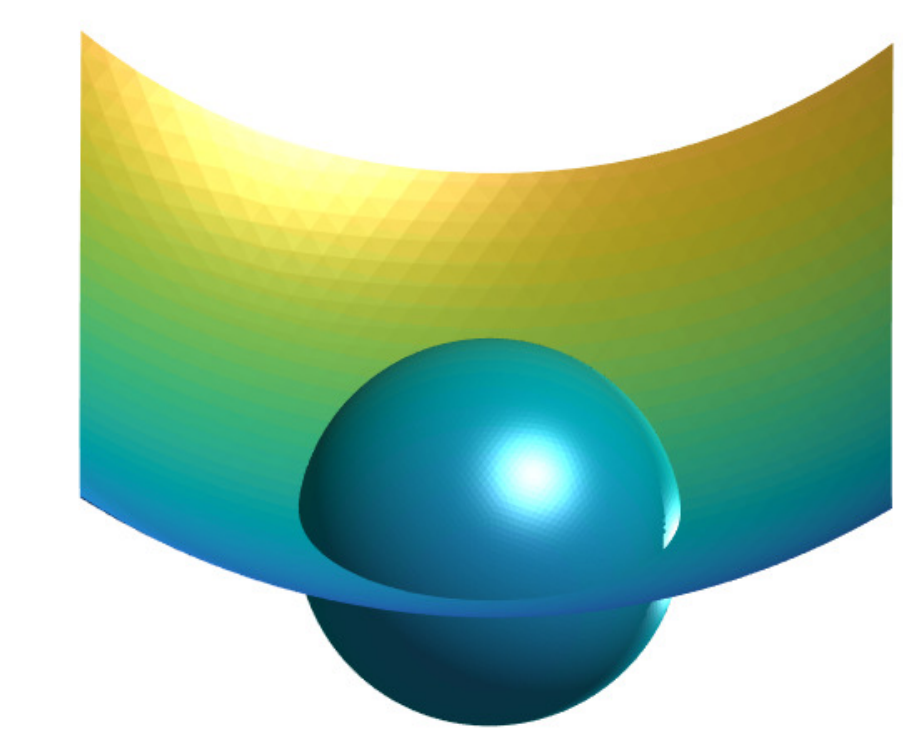}\label{fig:SA}}
\caption{Illustration of the circular area and spherical volume invariants.}
\end{figure}
  
Similarly, given a closed surface $S \subset \Rx3$ bounding a domain $\Omega = \intS$, we define the \emph{spherical volume invariant} at each point $p\in S$ to be the volume of the solid region given by intersecting the interior of the surface with a sphere of radius $r >0$ centered at the point $p$:
\begin{equation}\label{eq:Vsr}
V_{S,r}(p) = V(\Omega \capz B_r(p)  ),
\end{equation}
as illustrated in Figure \ref{fig:SA}. Again, invariance under three-dimensional Euclidean motions is clear. For $\C^3$ surfaces, the spherical volume invariant is related to the mean curvature of the surface via the expansion
\begin{equation}\label{eq:Vsrasym}
V_{S,r}(p) = \frac{2}{3}\pi r^3 - \frac{1}{4}\pi H(p) r^{4} + \O(r^5)\ \ \text{ as }r \to 0,
\end{equation}
where $H(p)$ is the mean curvature of $S$ at $p$, and $S$ is a $\C^3$ surface.  The spherical volume invariant has proven useful for feature extraction \cite{yang2006robust,pottmann2009integral,pottmann2007principal}, and a further analysis of the shape of the region $\intS \capz B_r(p)$ provides a robust estimation of the second fundamental form --- see \cite{pottmann2007principal} and Subsection \ref{sec:pk}.  Again, one can produce the corresponding global spherical volume invariant by integrating the local invariant over the entire surface. 

These quantities clearly extend to the corresponding hyperspherical volume invariant of closed hypersurfaces in $S \subset \R^n$.  
Our main result is the general formula \eqref{eq:Vform} that expresses this integral invariant in terms of a hypersurface integral over $S$.  In the planar case, with $n=2$, our general formula reduces to a useful formula  \eqref{eq:caiformula} or \eqref{eq:Vform2} for the circular area invariant $A_{C,r}(p)$ in terms of a suitable line integral over the curve $C$. For surfaces in $n=3$ dimensional space, it reduces to  the key formula \eqref{eq:Vform3}  for the spherical volume invariant $V_{S,r}(p)$ in terms of a surface integral over $S$. Our results apply to Lipschitz codimension $1$ submanifolds, which allows $S$ to be a triangulated mesh, as is often used to approximate surfaces in practice.  These new formulas are simple and fast to implement on a triangulated mesh. In particular, our method does not require discretizing the ambient three dimensional space off the surface, as was done using octrees and the Fast Fourier Transform in \cite{pottmann2009integral}. Similar ideas can be used to evaluate other integral invariants, although a number of them are already expressed in terms of integrals of the type sought after here.

This paper was motivated by an ongoing project to analyze and reassemble broken bone fragments, a problem of significant interest in anthropology, paleontology, and surgery, building on earlier work of two of the authors on planar and surface jigsaw puzzle reassembly, \cite{HOpuzzle, HumDum}.  A recent undergraduate REU project, \cite{Robp}, has successfully applied the circular area integral invariant to planar jigsaw puzzle reassembly, following \cite{HOpuzzle}.  Indeed, one can easily envision modifying the circular area invariant in order to incorporate designs  (writing, pictures, texture) that may appear on the puzzle pieces, potentially relying on some form of digital inpainting algorithm, \cite{BeEsGi,BBCS,ChSh,ChShi}, to extend the design in the circular region on one side of the curve to the other, after which it could be compared to other potential matches, or, alternatively  use of texture information to effect the reconstruction, as advocated in \cite{SagErc,SagErco}.  

Another potential application of these invariants is the detection of fracture edges, meaning ridges delineating the boundaries between original surfaces of the bone and break surfaces. Paleoanthropologists and zooarchaeologists study human biological and behavioral evolution and are interested in fracture edges because they provide valuable information about the agent of fragmentation~\cite{coil2017new,garcia2006determinacion,pickering2005contribution}, which may be, for example, humans, large carnivores, trampling, geological processes, or hydraulic action~\cite{merritt2017diagnostic,karr2012actualistic}. Determining the agent of fragmentation is essential for reconstructing how archaeological sites were formed. Fracture edges can also be used to find bones that refit, which aids in the identification of taxa and skeletal elements of vertebrates found at sites~\cite{bartram1999explaining}.  We propose to detect fracture edges by thresholding the spherical volume invariant, and demonstrate by showing results of detecting fracture edges on bone fragments in Section \ref{sec:numerics}.

The circular area and spherical volume invariants are particular cases of the general theory of integral invariants, \cite{HannHick,MCHYS,PWHY}, which have also been successfully applied to a variety of image processing problems.  See  \cite{FengKoganKrim} for applications of the moving frame method to their classification and signature construction under basic group actions, \eg Euclidean and equi-affine geometries. Distance histograms underly the widely used methods of shape contexts, \cite{BeMaPu}, and shape distributions, \cite{OFCD}. Histograms based on various geometric invariants (lengths, areas, etc.) play a fundamental role throughout a broad range of modern image processing algorithms, including shape representation and classification, \cite{AKKS,SHB}, image enhancement, \cite{SHB,Sapiro}, the scale-invariant feature transform (SIFT) \cite{Lowe, PelWer}, its affine-invariant counterpart (ASIFT), \cite{ASIFT}, and object-based query methods, \cite{SaGuUl}. 

\subsection{Outline}

In Section \ref{sec:circ} we give a simple formula for the circular area invariant in terms of a line integral. In Section \ref{sec:Vsr} we study the spherical volume invariant, and show how to use the Divergence Theorem to convert the volume integral into a surface integral, yielding a new formula for the invariant. Furthermore, in Subsection \ref{sec:pk}, we show how to extend our methods to estimate the principal curvatures of the surface by adapting the methods based on Principal Component Analysis (PCA) on local neighborhoods developed in  \cite{pottmann2007principal}. Finally, in Section \ref{sec:numerics} we discuss numerical implementations and present the results of numerical experiments on real data. We use the Euclidean norm on $\R^n$ throughout, leaving the investigation of more general norms to a future project.

\section{The Circular Area Invariant}
\label{sec:circ}

As a warmup, we consider the local circular area invariant \eqref{eq:Cai}. We assume $C$ is the oriented boundary of an open bounded domain $\Omega\subset \R^2$ with Lipschitz boundary. Consider a point $p\in C$ with $p=(p_1,p_2)$.   Consider the vector field  
\[\V (x)= \frac{1}{2} \,(x - p) = \frac{1}{2} \,(x_1-p_1,x_2-p_2)\]
 and notice that $\diva  \V    = 1$. By the Divergence Theorem, we can express the circular area invariant as 
\Eq{green}
$$\eeq{{A_{C,r}(p)= \iint_{\Omega \capz D_r(p) } \, dxdy}=\iint_{\Omega \capz D_r(p) } (\diva \V)  \, dxdy\\
= \oint_{C\capz D_r(p)}\V \cdot \nu\,ds + \oint_{\Omega\capz \partial D_r(p)}\V \cdot \nu\,ds,}$$
where $\nu$ denotes the unit outward normal to the curve $C$ in the first line integral and to the circular boundary $\partial D_r(p)$ in the second.  
Let us parametrize the circular boundary of the disk $D_r(p)$ by $c(\theta) = p + r(\cos \theta,\sin\theta)$, so that 
$$\V \cdot \nu = \frac{1}{2}\,r^2 \roq{on}\partial D_r(p).$$
 Let $\theta_1 < \theta_2$ be the angles at which the curve $C$ intersects the disk $D_r(p)$, assuming for the moment there are only 2 intersections and that $C$ lies inside the disk $D_r(p)$ for $\theta_1 < \theta < \theta_2$. Now, the second term in \eqref{green} is 
\begin{equation}\label{eq:second}
\oint_{\Omega\capz \partial D_r(p)}\V \cdot \nu\,ds = \int_{\theta_1}^{\theta_2} \frac{r^2}{2}(\sin^2\theta + \cos^2\theta)\, d\theta = \frac{r^2}{2}(\theta_2 - \theta_1).
\end{equation}
Therefore, our formula for the circular area invariant is
\begin{equation}\label{eq:caiformula}
A_{C,r}(p)=\oint_{C\capz D_r(p)}\V \cdot \nu\,ds+\frac{r^2}{2}(\theta_2-\theta_1).
\end{equation}
Notice this only involves integration along the curve $C$. The contour integral is a correction from the flat setting where $C$ is a line and $A_{C,r}(p)=\pi r^2/2$, since in this case $\theta_2-\theta_1 =\pi$ and $\V \cdot \nu = 0$ on $C\capz D_r(p)$. 

It is straightforward to generalize \eqref{eq:caiformula} to more than two intersections of $C$ and $\partial D_r(p)$. If the intersections occur at angles $\theta_1 < \theta_2 < \cdots < \theta_{2k}$, and $C$ lies inside the disk\footnote{We ignore any intersection point where, nearby, $C$ remains on one side or the other of the boundary of the disk.} $D_r(p)$ for $\theta_{2i-1} < \theta < \theta_{2i}$ for $i = 1,\ldots,k$. Then we have
\begin{equation}\label{eq:caiformula2}
A_{C,r}(p)=\oint_{C\capz D_r(p)}\V \cdot \nu\,ds+\frac{r^2}{2}\sum_{i=1}^k(\theta_{2i}-\theta_{2i-1}).
\end{equation}

\section{The Spherical Volume Invariant}
\label{sec:Vsr}

Having established a formula in the simple case of the two dimensional circular area invariant, we now turn to the spherical volume invariant \eqref{eq:Vsr}. The argument used in Section \ref{sec:circ} is not practical in three dimensions, since the integration over $\Omega\cap \partial D_r(p)$ in \eqref{green} becomes a surface integral, which 
defeats the point of reducing the calculation to an integral on the boundary surface.

We thus take a slightly different approach. Since the resulting formula will be applicable in all dimensions  $n\geq 2$, we proceed in general. We assume our hypersurface $S\subset \R^n$ is the boundary of an open and bounded set $\Omega\subset \R^n$ with Lipschitz boundary. Without loss of generality, we take $p=0\in S$, and set $B_r=B_r(p)=B_r(0)$ to be the ball of radius $r$ centered at $p=0$.  The hyperspherical invariant at $p=0$  is thus
\begin{equation}\label{eq:svi}
V_{S,r}:= V_{S,r}(0) =\int_{\Omega \capz B_r } \, dx.
\end{equation} 
Define the vector field
\begin{equation}\label{eq:v}
\V (x) = \frac{1}{n}\,x, \roq{and note that} \diva \V =1.
\end{equation}
For any divergence free vector field $\W \colon \R^n\to \R^n$, whereby $\div \W = 0$, we can express $V_{S,r}$ via the Divergence Theorem as
\begin{align}\label{eq:svi_alt}
V_{S,r}&=\int_{\Omega \capz B_r } \div(\V +\W )\, dx
&=\int_{S \capz B_r }(\V +\W )\cdot \nu \, dS + \int_{\partial \Omega \capz B_r }(\V +\W )\cdot \nu \, dS,
\end{align}
where $\nu$ denotes the outward normal to $S \capz B_r $ in the first term, and to $\partial \Omega \capz B_r $ in the second. The first term is an integral over the surface $S$, as we seek, while the second is an integral over $\partial B_r$, which is undesirable. 

Now, the idea is to choose the vector field $\W $ so that the second term vanishes, yielding our formula. Noting that $\V \cdot \nu = r/n$ on $\partial B_r$, we see that $\W $ must satisfy
\begin{equation}\label{eq:Veq}
\W \cdot \nu + \frac{r}{n} = 0 \rox{on } \partial \Omega \capz B_r .
\end{equation}
We will construct $\W $ as $\W  = \nabla u$ for a harmonic function $u$. Then \eqref{eq:Veq} is equivalent to the Poisson problem
\begin{equation}\label{eq:poisson}
\left\{\begin{aligned}
\Delta u &= 0&&\text{in }B_r,\\ 
\frac{\partial u}{\partial \nu} + \frac{r}{n} &=0&&\text{on } \ \partial \Omega \capz B_r .
\end{aligned}\right.
\end{equation}
If we look for a smooth solution of \eqref{eq:poisson} then the compatibility condition 
\begin{equation}\label{eq:comp}
\int_{\partial B_R}\frac{\partial u}{\partial \nu}\, dS = 0
\end{equation}
must hold. This would require modifying the boundary condition away from $\partial \Omega \capz B_r $, which is impractical, since the set $\partial B_r\setminus \Omega$ could be arbitrarily small, and is dependent on the particular point $p$ chosen on the surface.

Instead of seeking to satisfy the compatibility condition \eqref{eq:comp}, we relax the requirement that $u$ is smooth but continue to impose the boundary condition in \eqref{eq:poisson}.
We allow $u$ to have a singularity at the origin, and thus consider the Poisson problem 
\begin{equation}\label{eq:poissonorig}
\left\{\begin{aligned}
\Delta u &= 0&&\text{in }B_r\setminus\{0\},\\ 
\frac{\partial u}{\partial\nu} + \frac{r}{n} &=0&&\text{on } \partial B_r,
\end{aligned}\right.
\end{equation}
on the punctured ball.  A solution to the latter boundary value problem is given by 
\begin{equation}\label{eq:udef}
u(x) = \alpha_nr^n\Phi(x),
\end{equation}
where $\alpha_n$ is the measure of the unit ball in $\R^n$, and 
\begin{equation}\label{eq:fund}
\Phi(x) = 
\begin{cases}\dsty
-\,\frac{1}{2\pi}\log|x|,&\text{if }n=2\\\dsty
\frac{1}{n(n-2)\alpha_n\,|x|^{n-2}},&\text{if }n\geq 3
\end{cases}
\end{equation}
 is the fundamental solution of Laplace's equation. Thus, we are effectively circumventing the compatibility condition \eqref{eq:comp} by placing a point source at the origin. Due to the singularity of $u$, the argument leading to \eqref{eq:svi_alt} is no longer valid, and we need to proceed more cautiously.

First, we note that, for any $n$,
\begin{equation}\label{eq:fundprop}
\nabla u(x) = -\,\frac{r^n}{n}\,\frac{x}{|x|^{n}} \rox{for }x\neq 0.
\end{equation}
Let $0 < \eps < r$. By the Divergence Theorem and the boundary condition in \eqref{eq:poissonorig} we have 
\begin{align*}
\int_{S\capz(B_r\setminus B_\eps)}(\V+\nabla u)\cdot \nu \, dS &=\int_{\partial(\Omega\capz(B_r\setminus B_\eps))}(\V + \nabla u)\cdot \nu\, dS + \int_{\Omega\capz \partial B_\eps}(\V + \nabla u)\cdot \nu \, dS\\
&=\int_{\Omega\capz(B_r\setminus B_\eps)}\div(\V + \nabla u)\, dx + \int_{\Omega\capz \partial B_\eps}\left(\frac{\eps}{n}-\frac{r^n}{n\,\eps^{n-1}}\right) dS\\
&=\int_{\Omega\capz(B_r\setminus B_\eps)} dx + \left(\frac{\eps}{n}-\frac{r^n}{n\,\eps^{n-1}}\right)\H^{n-1}(\Omega\capz \partial B_\eps)\\
&=V_{S,r} - V_{S,\eps}+ \left(\frac{\eps}{n}-\frac{r^n}{n\,\eps^{n-1}}\right)\H^{n-1}(\Omega\capz \partial B_\eps),
\end{align*}
where $\H^{n-1}$ denotes $(n-1)$--dimensional Hausdorff measure. Therefore
\begin{equation}\label{eq:Vf}
V_{S,r} = V_{S,\eps} + \frac{1}{n}\int_{S\capz (B_r\setminus B_\eps)}\left(1-\frac{r^n}{|x|^n}\right)(x\cdot \nu) \, dS +\alpha_n(r^n-\eps^n)\, \frac{\H^{n-1}(\Omega\capz \partial B_\eps)}{\H^{n-1}(\partial B_\eps)}.
\end{equation}
All that is left is to send $\eps\to 0$, and we state the consequence as a theorem.
\begin{theorem}\label{thm:Vsr}
Let $\Omega\subset \R^n$ be open and bounded with Lipschitz boundary $S:=\partial \Omega$. Let $p\in S$ and assume the limit
\begin{equation}\label{eq:Gamma}
\Gamma(p):= \lim_{\eps\to 0^+}\frac{\H^{n-1}(\Omega\capz \partial B_\eps(p))}{\H^{n-1}(\partial B_\eps(p))}
\end{equation}
exists. Then we have
\begin{equation}\label{eq:Vform}
V_{S,r}(p) = \frac{1}{n}\int_{S\capz B_r(p)}\left(1-\frac{r^n}{|x-p|^n}\right)(x-p)\cdot \nu \, dS + \alpha_nr^n \Gamma(p) .
\end{equation}
\end{theorem}

A few remarks are in order.

\begin{remark}
Notice the integrand in \eqref{eq:Vform} has a singularity at $x=p$. Since $S$ is only assumed  to be Lipschitz, the singularity may not be integrable, and so we define the integral via its principal value
\[\int_{S\capz B_r(p)}\left(1-\frac{r^n}{|x-p|^n}\right)(x-p)\cdot \nu \, dS:=\lim_{\eps\to 0^+}\int_{S\capz (B_r(p)\setminus B_\eps(p))}\left(1-\frac{r^n}{|x-p|^n}\right)(x-p)\cdot \nu\, dS,\]
which exists thanks to \eqref{eq:Vf} and \eqref{eq:Gamma}. If $S\in \C^{1,\alpha}$ then we have 
\[(x-p)\cdot \nu = \O(|x-p|^{1+\alpha}) \text{ as }x\to p\]
and so the kernel singularity $|x-p|^{1-n+\alpha}$ is integrable on the $n-1$ dimensional surface. 
\label{rem:PV}
\end{remark}

\begin{remark}
If the surface $S$ is differentiable at $p$ then $\Gamma(p) = \frac{1}{2}$, and thus 
\begin{equation}\label{eq:VformC1}
V_{S,r}(p) = \frac{1}{n}\int_{S\capz B_r(p)}\left(1-\frac{r^n}{|x-p|^n}\right)(x-p)\cdot \nu \, dS + \frac{1}{2}\,\alpha_nr^n.
\end{equation}
Since a Lipschitz surface is differentiable almost everywhere, the formula \eqref{eq:VformC1} holds at almost every point of $S$. 
\label{rem:Gamma1}
\end{remark}

\begin{remark}
If the surface $S \subset \Rx3$  is a triangulated mesh and $p\in S$ is a vertex of the mesh, then 
\[(x-p)\cdot \nu = 0\]
at all points $x$ in the vertex polygon associated to $p$ (i.e., the triangles adjacent to $p$), and where $\nu$ denotes the unit normal to the triangle containing $x$.  Thus, the kernel is integrable on triangulated meshes.
Moreover, $\Gamma(p)$ exists, and \eqref{eq:Vform} holds,  for every $p\in S$. In Subsection \ref{sec:gamma}, we derive an explicit formula for $\Gamma(p)$ on a triangulated mesh in terms of the vertex polygon of $p$. 
\label{rem:Gamma2}
\end{remark}
\begin{remark}
The limit \eqref{eq:Gamma} defining $\Gamma(p)$ may fail to exist at a point of non-differentiability of a Lipschitz hypersurface $S$. Consider, for example, $n=2$ and take the curve $C$ to be the graph of the Lipschitz function 
\[f(x) = |x|\sin\bpa{\log |x| }.\]
Take the interior of $C$ to be the epigraph $\{x\in \R^2 \, :\, f(x)>0\}$. Then the limit \eqref{eq:Gamma} does not exist at $p=0$, since along the sequence $x_k=e^{\frac{\pi}{2}-k\pi}$ we have $f(x_k)=f(x_k)=(-1)^kx_k$. 
\label{rem:Gamma3}
\end{remark}
\begin{remark}
Finally, let us note that in dimension $n=2$, the formula \eqref{eq:Vform} reads
\begin{equation}\label{eq:Vform2}
A_{C,r}(p) = \frac{1}{2}\oint_{C\capz D_r(p)}\left(1-\frac{r^2}{|x-p|^2}\right)(x-p)\cdot \nu\, ds + \pi \, r^2\,\Gamma(p) .
\end{equation}
In dimension $n=3$, it becomes
\begin{equation}\label{eq:Vform3}
V_{S,r}(p) = \frac{1}{3}\int_{S\capz B_r(p)}\left(1-\frac{r^3}{|x-p|^3}\right)(x-p)\cdot \nu  \, dS + \frac{4}{3}\,\pi \,r^3\,\Gamma(p) .
\end{equation}
\label{rem:n2n3}
\end{remark}

\subsection{An analytic expression for \texorpdfstring{$\Gamma(p)$}o on a triangulated mesh}
\label{sec:gamma}

We give here an analytic expression for $\Gamma(p)$, defined in \eqref{eq:Gamma}, when $p$ is a vertex of a triangulated mesh surface in $\R^3$. Let us assume we have made a translation and rotation so that the vertex under consideration is $p=0$ and the unit outward normal vector at the origin is $\nu = (0,0,-1)$. Of course, there is no well-defined normal at the vertex $p=0$ itself, and so $\nu$ should chosen to be ``close'' to the nearby unit normals, in that it approximates the normal to the smooth surface represented by the mesh. For example, it could be the normalized average of the normals to the triangles in the vertex polygon; another possibility is that it is the normal to the least squares approximating plane to the vertices adjacent to $p$.

The computation of $\Gamma:=\Gamma(0)$ involves only the  \emph{vertex triangles}  $T_1,\dots,T_k$ that are adjacent to the vertex $p=0$. See Figure \ref{fig:gamma} for a depiction of these triangles and the area of the sphere we wish to compute. Since the outward normal at $p=0$ is $(0,0,-1)$, we will also   assume that the outward unit normal vector $\nu^i=(\nu^i_1,\nu^i_2,\nu^i_3)$ to each vertex triangle $T_i$ satisfies\footnote{We will exclude ``bizarre'' vertices where this assumption does not hold under any reasonable choice of the normal $\nu$ at $p$.} $\nu^i_3<0$.  Finally, in view of the definition \eqref{eq:Gamma} of $\Gamma $, we may extend the vertex triangles to $\infty$ in the radial direction, and compute
\begin{equation}\label{eq:gamma2}
\Gamma:= \frac{1}{4\pi}\int_{\Omega\capz \partial B_1} \, dS,
\end{equation}
where $\Omega$ is the region above the (extended) vertex triangles in the $x_3$-direction.
\begin{figure}
\centering
\subfloat[Vertex triangles]{\includegraphics[trim = 30 30 30 30, clip = true, width=0.32\textwidth]{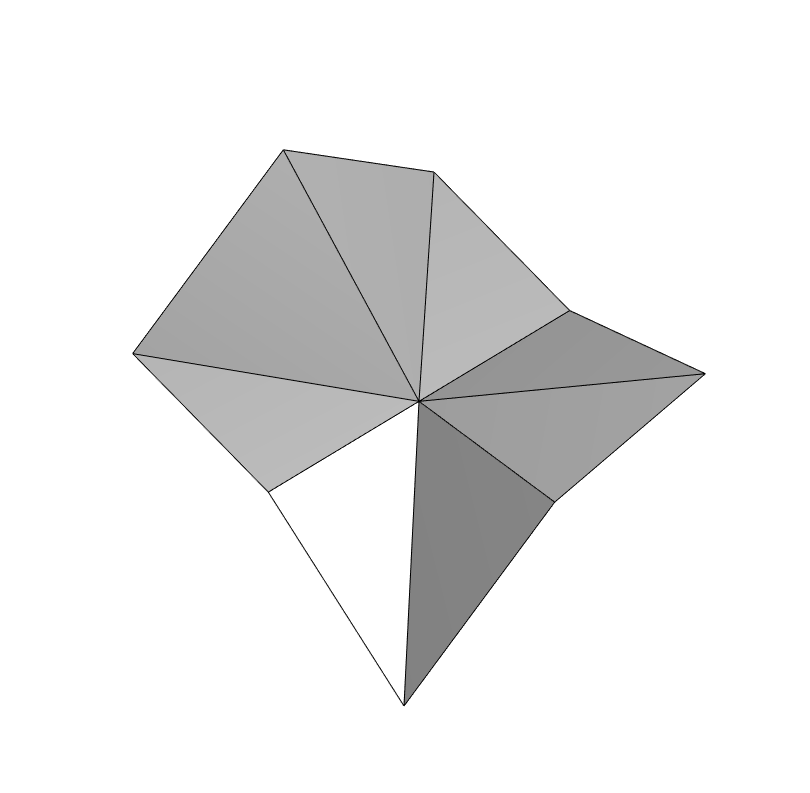}}
\subfloat[Small sphere]{\includegraphics[trim = 30 30 30 30, clip = true, width=0.32\textwidth]{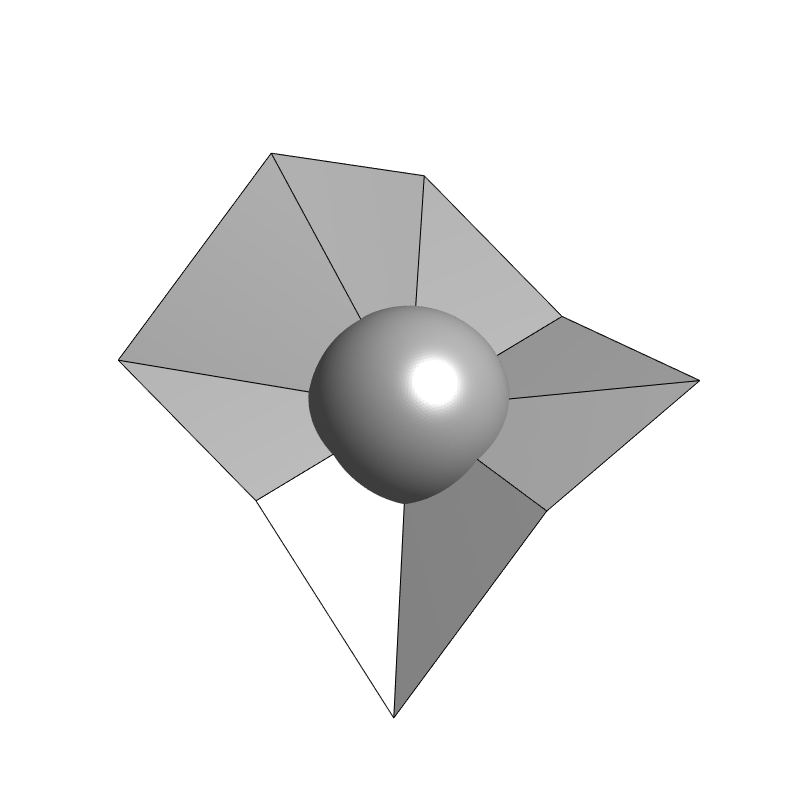}}
\subfloat[From above]{\includegraphics[width=0.32\textwidth]{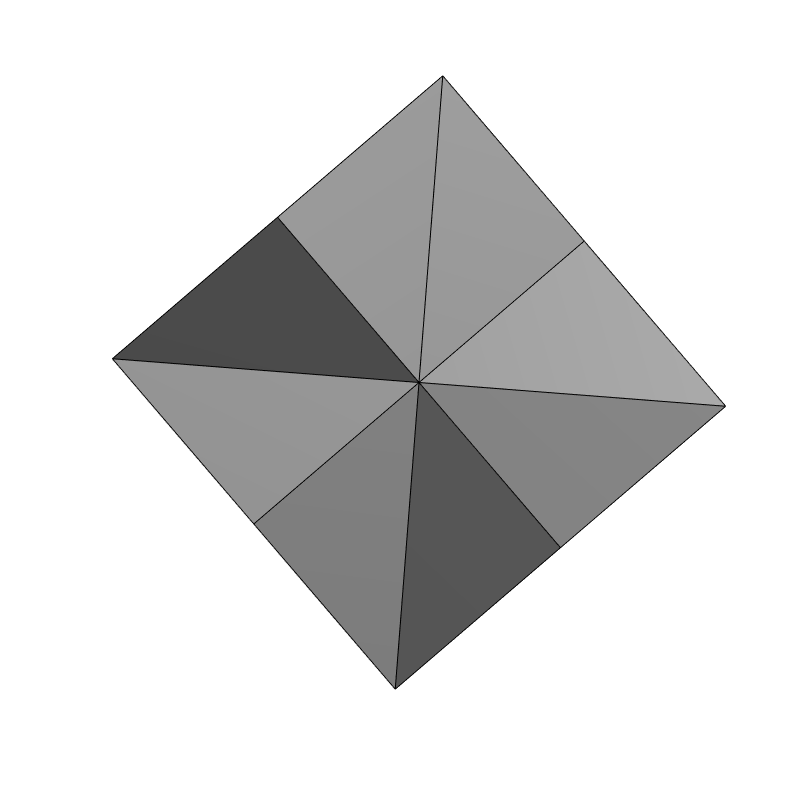}}
\caption{Depiction of the vertex triangles used in the computation of $\Gamma$. To compute $\Gamma$, we need to compute the fraction of surface area of the sphere in (B) that lies above the mesh. }
\label{fig:gamma}
\end{figure}

We work in spherical coordinates
\begin{equation}\label{eq:spherical}
x_1 = r\sin\vph \;\cos\tha, \qquad  x_2 = r\sin\vph \;\sin \tha, \qquad  x_3 = r\cos\vph.
\end{equation}
The edges $E_1,\ldots, E_k$ of the vertex triangles $T_1,\ldots, T_k$ containing the origin will be called \emph{vertex edges}, and we let  $(\tha_i,\vph_i)$ be their corresponding spherical angles.
We order the vertex edges and triangles so that \[0 \leq \tha_1 < \tha_2 < \tha_3 < \cdots < \tha_k < 2\pi,\]
and, for convenience, set $E_{k+1} = E_1$ with azimuthal angle $\tha_{k+1} = \tha_1 + 2\pi \geq 2\pi$. 
The vertex triangles are similarly ordered, so that $T_i$ has vertex edges $E_i$ and $E_{i+1}$.  Observe that the vertex edges  of $T_i$ are ordered so that $E_i,E_{i+1},\nu_i$ form a left-handed frame, keeping in mind that, by the preceding assumption, the normal $\nu_i$ points \emph{downwards}.

Each vertex triangle $T_i$ intersects the unit sphere $S_1 = \partial B_1$ along a curve 
\begin{equation}\label{eq:Ci}
C_i = \{\vph=h_i(\tha)\} = T_i \,\cap \, S_1
\end{equation}
connecting $(\tha_i,\vph_i)$ to $(\tha_{i+1},\vph_{i+1})$. In terms of the intersection curves $C_i$ we can compute
\begin{equation}\label{eq:gamma_formula1}
\Gamma = \frac{1}{4\pi}\sum_{i=1}^{k} \int_{\tha_i}^{\tha_{i+1}}\int_0^{h_i(\phi)}\sin\vph \, d\vph \,d\tha = \frac{1}{2} - \frac{1}{4\pi}\sum_{i=1}^{k} \int_{\tha_i}^{\tha_{i+1}}g_i(\tha) \, d\tha, \end{equation}
where
$$g_i(\tha) = \cos h_i(\tha) .$$
We can find an explicit formula for $g_i(\tha)$. Indeed, the face $T_i$ is described by the plane
\begin{equation}\label{eq:Tigraph}
x_3 = a_i x_1 + b_ix_2, \rox{where}a_i = -\,\frac{\nu^i_1}{\nu^i_3} \rox{and} b_i = -\,\frac{\nu^i_2}{\nu^i_3}.
\end{equation}
Therefore, the intersection curve \eqref{eq:Ci}
satisfies
\[g_i(\tha) = \cos h_i(\tha) = (a_i\cos\tha + b_i \sin\tha) \sin h_i(\tha),\]
and hence
\begin{equation}\label{eq:hphi}
h_i(\tha) = \coti\left( a_i\cos\tha + b_i\sin\tha \right).
\end{equation}
Noting the identity
\[\cos(\coti x) = \frac{x}{\sqrt{1 + x^2}},\]
we have
\begin{equation*}
g_i(\tha) =\frac{a_i\cos\tha + b_i\sin\tha}{\sqrt{1 + (a_i\cos\tha + b_i\sin\tha)^2}}.
\end{equation*}
Since 
$$a_i\cos\tha + b_i\sin\tha = c_i \cos(\tha - \delta _i), \roq{where}\delta_i =\atan{b_i}{a_i}, \quad c_i = \sqrt{a_i^2 + b_i^2},$$
 we can simplify the preceding formula to read
\begin{equation}\label{eq:zi}
g_i(\tha) =\frac{c_i\cos(\tha - \delta_i)}{\sqrt{1 + c_i^2\cos^2(\tha-\delta_i)}}.
\end{equation}
We note that $\atan{y}{x}$ is the two-argument $\arctan$ function, which gives the angle in radians between the positive $x$-axis and the ray from the origin to the point $(x,y)$, returning values in the interaval $[0,2\pi)$.
Integrating $g_i$ yields
\begin{equation}\label{eq:zi_anti}
\int g_i(\tha) \, d\tha = \arcsin\left(d_i\sin(\tha-\delta_i) \right) + \text{Constant},\rox{where }d_i =\frac{c_i}{\sqrt{1+c_i^2}}.
\end{equation}
This yields the following explicit formula:
\begin{equation}\label{eq:gamma_final}
\ \Gamma = \frac{1}{2} - \frac{1}{4\pi}\sum_{i=1}^{k}  \;\bbk{\arcsin(d_i\sin(\tha_{i+1}-\delta_i)) - \arcsin(d_i\sin(\tha_i-\delta_i))}.\ 
\end{equation}
Unwrapping the definitions we have
\[1 + c_i^2 = 1 + a_i^2 + b_i^2 = 1 + \frac{(\nu^i_1)^2}{(\nu^i_3)^2}+ \frac{(\nu^i_2)^2}{(\nu^i_3)^2}=\frac{1}{(\nu^i_3)^2},\roq{so} c_i^2 =\frac{(\nu^i_1)^2 + (\nu^i_2)^2 }{(\nu^i_3)^2}.\]
 It follows that
\begin{equation}\label{eq:di}
d_i = \sqrt{(\nu^i_1)^2 + (\nu^i_2)^2}.
\end{equation}
We also note that 
\begin{equation}\label{eq:alphai}
\delta_i = \atan{\nu^i_2}{\nu^i_1}, \qquad \tha_i = \atan{y_i}{x_i},
\end{equation}
where $(x_i,y_i,z_i)$ is any point along the edge $E_i$.

\subsection{Principal component analysis on local neighborhoods}
\label{sec:pk}

The spherical volume invariant of a surface in $\R^3$ is a robust estimator of its mean curvature, due to the asymptotic expansion given in  \eqref{eq:Vsrasym}. However, it gives no information about other differential geometric quantities of interest,
such as the second fundamental form, the individual principal curvatures,
the Gauss curvature,
or the directions of principal curvature.

To capture additional geometric information, we follow \cite{pottmann2007principal} and analyze the shape of the region $\Omega \capz B_r(p) $. In particular, it is suggested in \cite{pottmann2007principal} to perform principal component analysis (PCA) on this region, that is, we compute the eigenvalues $\lambda_1(p)\geq \lambda_2(p)\geq \lambda_3(p)$ of the $3 \times 3$ symmetric matrix\footnote{Here we take $x$ to be a column vector.}
\begin{equation}\label{eq:covM}
M_{S,r}(p) := \int_{\Omega \capz B_r(p) }(x-\bar{x}(p))(x-\bar{x}(p))^T \, dx,
\end{equation}
where 
\begin{equation}\label{eq:centroid}
\bar{x}(p) := \frac{1}{V_{S,r}(p)}\int_{\Omega \capz B_r(p) }x \, dx
\end{equation}
 is the centroid of $\Omega \capz B_r(p) $, cf.~\eqref{eq:Vsr}.  
Assuming $S$ is sufficiently smooth, it was shown in \cite{pottmann2007principal} that the eigenvalues  of $M_{S,r}(p)$ have the asymptotic expansions
\Eq{eq:eigenexpansions}
$$\eeq{\lambda_1(p) = \frac{2\pi}{15}r^5 - \frac{\pi}{48}\bbk{ 3\kappa_1(p) + \kappa_2(p) }r^6 + \O(r^7)\\
\lambda_2(p)= \frac{2\pi}{15}r^5 - \frac{\pi}{48}\bbk{ \kappa_1(p) + 3\kappa_2(p) }r^6 + \O(r^7)\\
\lambda_3(p) = \frac{19\pi}{480}r^5-\frac{9\pi }{512}\bbk{ \kappa_1(p) + \kappa_2(p) }r^6 + \O(r^7),}\roq{as \ \ $r \longrightarrow 0$,} 
$$
where $\kappa_1(p),\kappa_2(p)$ are the principal curvatures of the surface $S$ at the point $p\in S$, and, in the last formula, the $\O(r^6)$ term gives the mean curvature
$$H(p)=\tfrac{1}{2}\bbk{\kappa_1(p) + \kappa_2(p)}.$$  
Moreover, the first two corresponding eigenvectors $\v_1,\v_2$ are approximately tangent to the surface, and, assuming we are at a non-umbilic point, offer an $\O(r/|\kappa_1-\kappa_2|)$ approximation of the directions of principal curvatures, while $\v_3$ is approximately normal to the surface and is an $\O(r^2)$ approximation of the unit normal.  Thus, the matrix $M_{S,r}(p)$ provides a robust estimation of the second fundamental form of $S$ at a non-umbilic point $p$. 

Let us now show how to compute the matrix $M_{S,r}(p)$ via surface integrals, as we did for the spherical volume invariant $V_{S,r}(p)$ in Theorem \ref{thm:Vsr}. While these results are mainly of interest in dimension $n=3$, we carry out the derivation for an arbitrary dimension $n$. Noting that 
\begin{equation}\label{eq:covM2}
M_{S,r}(p)= \int_{\Omega \capz B_r(p) }(x-p)(x-p)^T \, dx - V_{S,r}(p)(\bar{x}(p)-p)\,(\bar{x}(p)-p)^T.
\end{equation}
it suffices to compute the first two moments
\begin{equation}\label{eq:moments}
m_i(p):=\int_{\Omega \capz B_r(p) } (x_i - p_i)\, dx , \qquad  c_{ij}(p):=\int_{\Omega \capz B_r(p) } (x_i-p_i)(x_j-p_j) \, dx,
\end{equation}
in terms of which the $(i,j)$ entry of $M_{S,r}(p)$ is given by
\begin{equation}\label{eq:Msr}
[M_{S,r}(p)]_{i,j} = c_{ij}(p) - \frac{1}{V_{S,r}(p)}m_i(p)\,m_j(p).
\end{equation}
The computation of $m_i(p)$ and $c_{ij}(p)$ in terms of surface integrals is relatively straightforward, compared to the computation of $V_{S,r}$. In what follows, $e_1,e_2,\dots,e_n$ denote the standard basis vectors in $\R^n$, and $\delta_{ij}$ is the Kronecker delta.

\begin{lemma}\label{lem:mi}
Let us abbreviate $y=x-p$.
Then, for any $1\leq i,j\leq n$, we have
\begin{equation}\label{eq:mi}
m_i(p) = \frac{1}{n+1}\int_{S \capz B_r(p) } (y_i y - r^2e_i)\cdot \nu\, dS(x).
\end{equation}
and
\begin{equation}\label{eq:cij}
c_{ij}(p) = \frac{r^2}{n+2} V_{S,r}(p)\delta_{ij} + \frac{1}{2n+4}\int_{S \capz B_r(p) } (2y_iy_j y - r^2(y_je_i + y_ie_j))\cdot \nu\, dS(x).
\end{equation}
\end{lemma}

\begin{proof}
Without loss of generality, we may assume $p=0$. Then $y=x$ and we write $B_r=B_r(p)=B_r(0)$. We first prove \eqref{eq:mi}. Define the vector field 
$$\V(x) = \frac{x_ix - r^2e_i}{n+1} \roq{so that}\diva\V = x_i.$$
 By the Divergence Theorem,
\[m_i = \int_{\Omega \capz B_r } \diva\V\, dx = \int_{S \capz B_r } \V(x)\cdot \nu\, dS + \int_{\partial \Omega \capz B_r } \V(x)\cdot \nu \, dS.  \]
On the spherical portion of the boundary $\partial \Omega \capz B_r $, we have $\nu =x/r$ and  so
\[\V(x)\cdot \nu = \frac{1}{r}\,\V(x)\cdot x = \frac{x_i(|x|^2 - r^2)}{(n+1)\,r} = 0\]
since $|x|^2=r^2$ on $\partial B_r$. This completes the proof of \eqref{eq:mi}.

We now prove \eqref{eq:cij}. Define the vector field 
$$\W(x) = \frac{2x_ix_j x - r^2(x_je_i + x_ie_j)}{2n+4},\roq{whereby} 
\diva\W = x_ix_j - \frac{1}{n+2}\,r^2 \delta_{ij}.$$
By the Divergence Theorem, we have
\begin{align*}
c_{ij} &= \int_{\Omega \capz B_r } \left(\frac{1}{n+2}\,r^2\,\delta_{ij} + \diva\W \right) dx \\
 &=\frac{1}{n+2}\,r^2\, \delta_{ij}V_{S,r} + \int_{S \capz B_r } \W(x)\cdot \nu\, dS + \int_{\partial \Omega \capz B_r } \W(x)\cdot \nu \, dS.  
\end{align*}
On the  portion of the boundary $x\in\partial \Omega \capz B_r $
\[\W(x)\cdot \nu = \frac{1}{r}\,\W(x)\cdot x = \frac{2x_ix_jr^2 - r^2(x_jx_i + x_ix_j)}{(2n+4)\,r} = 0,\]
which completes the proof.
\end{proof}

\section{Implementation}
\label{sec:implementation}

Let us next discuss how to compute the surface integrals from Theorem \ref{thm:Vsr} and Lemma \ref{lem:mi} on a surface given as a triangulated mesh, which is often the case in practice. The integrals we wish to compute all have the form
\begin{equation}\label{eq:compuc}
\int_{S\cap B_r(p)} f(x) \, dS
\end{equation}
for various choices of kernel function $f(x)$. We adopt the convention that $f(x)=0$ if $|x-p|>r$, and hence rewrite \eqref{eq:compuc} as simply
\begin{equation}\label{eq:compu}
\int_{S} f(x) \, dS.
\end{equation}

Let $T_1,\dots,T_M$ denote the triangles in the triangulated surface $S$. Then we can write
\begin{equation}\label{eq:split}
\int_{ S} f(x) \, dS =\sum_{m=1}^M \ \int_{ T_m}f(x) \, dS.
\end{equation}
We show in Sections \ref{sec:analytic} and \ref{sec:boundary} that the triangular integrals appearing in the summation can be computed analytically for all of the kernels $f$ used in this paper.
Let us note that on the right hand side of \eqref{eq:split}, we need only sum over triangles $T_m$ that have non-empty intersection with $B_r(p)$. However, it is computationally expensive to perform a range search to find all such triangles, especially for large meshes. In our implementation, we instead perform a depth first search on the triangle graph of the mesh, starting at any triangle adjacent to $p$, and terminating when all triangles in the \emph{connected component} of $S\cap B_r(p)$ containing $p$ are found. While the depth first search has linear complexity and is very fast in practice, it will fail to find any additional connected components of $S\cap B_r(p)$ that do not contain $p$. On the other hand, this may be a desirable property of the algorithm, especially if one is primarily interested in the \emph{local} geometry of the mesh.

\subsection{Analytic integration over triangles}
\label{sec:analytic}

Let us show how \emph{all} the integrals considered in this paper can be computed analytically over triangles $T_m\subset B_r(p)$. 
For simplicity, we take $p=0$, write $B_r=B_r(0)$, and consider a triangle $T$.

For the spherical volume invariant, for any triangle $T$ with $T\subset B_r$ the surface integral \eqref{eq:Vform} from Theorem \ref{thm:Vsr} requires us to compute
\[A:=\frac{1}{3}\int_{T}\left( 1-\frac{r^3}{|x|^3} \right)x\cdot \nu \, dS.\]
Since $x\cdot \nu$ is constant over the triangle $T$, we have
\[A = \frac{1}{3}\,z\cdot \nu\left( |T| - r^3\int_{T}\frac{1}{|x|^3}\, dS \right),\]
where $z$ is any point belonging to $T$, such as its centroid or one of its vertices, while $|T|$ denotes the surface area of $T$. The remaining integrand $|x|^{-3}$ is known as a \emph{hypersingular} kernel, and arises, for instance, in the boundary element method for solving partial differential equations \cite{banerjee1981boundary}.  The integral of this hypersingular kernel over any planar triangle can be computed analytically \cite{nintcheu2009explicit} provided $p=0\not\in T$, which we may freely assume since $z\cdot \nu=0$ when $0\in T$. For convenience, we recall the analytic formula, which is rather tedious and derived in \cite{nintcheu2009explicit}, in Appendix \ref{sec:hyper}.  

For PCA on local neighborhoods, the integrals we need to compute from Lemma \ref{lem:mi} correspond to
\begin{equation}\label{eq:miint}
\frac{1}{4}\int_{T}(x_ix - r^2 e_i)\cdot \nu \, dS, \roq{and}\frac{1}{10}\int_{T} \bbk{2x_ix_jx - r^2 (x_je_i+x_ie_j)}\cdot \nu\,  dS.
\end{equation}
Since $x\cdot \nu$ and $e_i\cdot \nu$ are constant over $T$, we just need to compute the quantities 
\begin{equation}\label{eq:quant}
a_i:=\int_T x_i\, dS, \qquad  b_{ij}:=\int_T x_ix_j \, dS.
\end{equation} 
Let us denote  the vertices of $T$ by $x,y,z\in \R^3$. The first integrand in \eqref{eq:quant} is linear, and so the integral can be computed analytically with the three point stencil
\begin{equation}\label{eq:ai}
a_i = \frac{1}{3}\,|T|\,(x_i + y_i + z_i).
\end{equation}
For $b_{ij}$, we compute the integral in barycentric coordinates
\begin{align*}
b_{ij} &=2\,|T|\int_0^1 \int_0^{1-t} ((1-s-t)x_i +sy_i + tz_i)((1-s-t)x_j +sy_j + tz_j)\, ds dt, \\
&=2\,|T|\Bigg[\int_0^1 \int_0^{1-t} (1-s-t)^2 x_ix_j + s^2y_iy_j + t^2 z_iz_j + st(y_iz_j + y_jz_i)\\
&\hspace{4cm}+ (1-s-t)s(x_iy_j + y_jy_i) + (1-s-t)t(x_iz_j + z_jz_i)\, dsdt\Bigg]
\end{align*}
Computing
\[\int_0^1\int_0^{1-t} (1-s-t)^2\, dsdt =\int_0^1\int_0^{1-t} s^2\, dsdt =\int_0^1\int_0^{1-t} t^2\, dsdt = \frac{1}{12},\]
and
\[\int_0^1\int_0^{1-t} (1-s-t)s\, dsdt =\int_0^1\int_0^{1-t} (1-s-t)t\, dsdt =\int_0^1\int_0^{1-t} st\, dsdt = \frac{1}{24},\]
we have
\begin{equation}\label{eq:bij}
b_{ij} = \frac{1}{12}|T|(2x_ix_j + 2y_iy_j + 2z_iz_j + x_iy_j + x_jy_i + x_iz_j + x_jz_i + y_iz_j + y_jz_i).
\end{equation}
If we were to denote the vertices by $v_i = (v_i^1,v_i^2,v_i^3)$, say, then \eqref{eq:bij} would have the simple form
$$b_{ij} = \frac{1}{12}\,|T| \sum_{p,q=1}^3 v_i^pv_j^q,$$
and similarly for \eqref{eq:ai}.

\subsection{Boundary triangles}
\label{sec:boundary}

For triangles $T$ that have a non-empty intersection with the boundary $\partial B_r$ of the ball $B_r$, the integral over $T$ cannot be computed analytically. To determine whether a triangle $T$ intersects $\partial B_r$, we compute 
$$r_1:=\min_{x\in T} |x|, \qquad r_2:=\max_{x\in T} |x|,$$
 and check whether $r_1\leq r \leq r_2$. To compute $r_2$, it is sufficient to check the vertices of the triangle, since $x\mapsto |x|$ is convex. The computation of $r_1$ is more tedious, since the minimum distance may occur interior to $T$. To compute $r_1$ we orthogonally project the origin $p=0$ onto the plane containing the triangle $T$, calling the projection $x_P$. If $x_P\in T$, then $r_1=|x_P|$. If $x_P\not\in T$, then we find the closest point $x_T\in T$ to the projection $x_P$, and therefore $r_1^2 = |x_T|^2 + |x_T-x_P|^2$ by the Pythagorean Theorem. 

To compute the integral over such boundary triangles, we fix a maximum desired side length $\ell>0$ and recursively bisect the triangle along the line segment connecting the midpoint of its longest side with the opposing vertex.  We stop the bisection procedure on a given subtriangle $T_s$ if $T_s\cap \partial B_r = \varnothing$, or the maximum side length of $T_s$, denoted $L(T_s)$, falls below $\ell$. See Figure \ref{fig:bisection} for an illustration of the bisection process. We compute the integration over $T_s$ analytically if $T_s\subset B_r$, or with the approximation 
\begin{figure}
\centering
\includegraphics[width=0.6\textwidth]{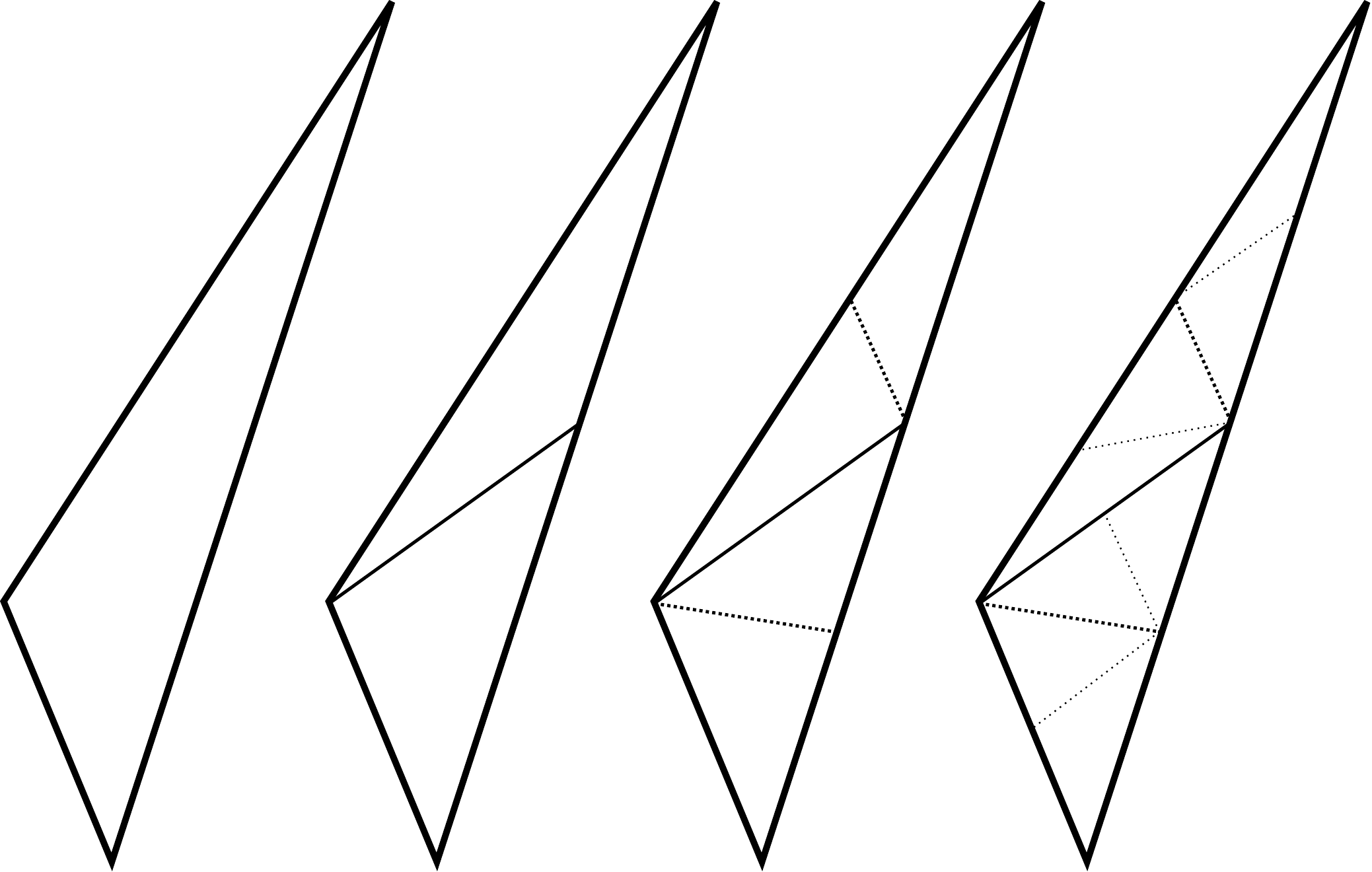}
\caption{Illustration of the bisection process. The triangle on the left is recursively bisected three times from left to right. Each bisection splits a triangle into two along the line segment between the midpoint of the longest side and the opposing vertex, generating two triangles of equal area. Hence, each sub-triangle on the right has exactly $1/8$ of the area of the original triangle.}
\label{fig:bisection}
\end{figure}

\begin{equation}\label{eq:approxbasic}
\int_{T_s} f\, dS \approx |T_s|\,f\left(\frac{x+y+z}3\right),
\end{equation}
if $T_s\cap \partial B_r \not= \varnothing$, where $x,y,z$ are the vertices of $T_s$. We note the approximation error is bounded by
\begin{equation}\label{eq:intapprox}
\left|\int_{T_s} f\, dS - |T_s|\,f\left(\frac{x+y+z}3\right)\right| \leq |T_s|\,\osc{T_s}f,
\end{equation}
where 
\[\osc{T_s} f := \max_{T_s}f - \min_{T_s} f\]
denotes the oscillation of the function $f$ over the triangle $T_s$.
Now, let $\delta>0$ so that 
\begin{equation}\label{eq:deltadef}
\text{osc}_{T_s} f \leq \delta \ \text{ whenever } T_s\cap \partial B_r\not= \varnothing \text{ and }L(T_s)\leq \ell.
\end{equation}
Note that $\delta$ is reduced throughout the bisection procedure, since the diameter of triangles that intersect $\partial B_r$ is decreasing.
Since the triangles $T_s$ in \eqref{eq:deltadef} belong to $S\cap B_{r+\ell}\setminus B_{r-\ell}$, the error in computing \eqref{eq:compu} is bounded by
\begin{equation}\label{eq:boundTs}
|S\cap B_{r+\ell}\setminus B_{r-\ell}|\,\delta,
\end{equation}
where $|S\cap A|$ denotes the surface area of $S\cap A$.
We assume that for the mesh $S$, there exists a constant $C>0$, independent of $r$ and $\ell$, such that 
\begin{equation}\label{eq:volgrowth}
|S\cap B_{r+\ell}\setminus B_{r-\ell}|\leq C\left[\,\pi(r+\ell)^2 - \pi(r-\ell)^2\,\right] = 2\,C\,\pi \,r\,\ell.
\end{equation}
Therefore, our approximation error is at most
\begin{equation}\label{eq:ferror}
\text{Integration error} \leq 2\,C\,\pi \,r\,\ell\,\delta,
\end{equation}
where $\delta$ is defined in \eqref{eq:deltadef}. We note the volume growth assumption \eqref{eq:volgrowth} is convenient, in that it leads to a simple form for the integration error \eqref{eq:ferror}. However, the analysis below can be easily carried out with other assumptions in place of \eqref{eq:volgrowth}, if needed. The volume growth assumption is true for smooth surfaces, with $r>0$ small, and hence for any triangulated mesh that well-approximates a smooth surface. 

The application of \eqref{eq:ferror} depends on the context. For the spherical volume invariant, we have 
\[f(x) = -\frac{1}{3}\left( \frac{r^3}{|x|^3}-1 \right)_+(x\cdot \nu),\]
where $a_+:=\max\{a,0\}$. For any triangle $T_s$ with maximum side length less than $\ell$ and satisfying $T_s\cap \partial B_r\not =\varnothing$, we have 
\[\osc{T_s} f \leq \frac{r^4\ell}{(r-\ell)^4}.\]
Thus, $\delta$ from \eqref{eq:deltadef} can be chosen as $\delta = r^4\ell/(r-\ell)^4$. Since the spherical volume invariant scales with $r^3$, it is reasonable to select an error tolerance $\eps>0$ and ask that the integration error is bounded by $\eps\,r^3$. Thus, invoking \eqref{eq:ferror} we find that $\ell$ should be selected so that $\ell < r$ and 
\begin{equation}\label{eq:lsvi}
\ell^{\, 2} \leq \frac{\eps(r-\ell)^4}{r^2}.
\end{equation}
Note that we are discarding the constant $2C\pi$ in \eqref{eq:volgrowth}, since we are only interested in how $\ell$ should scale with $r$ and $\eps$. If $\eps \ll 1$ so that $\ell\ll r$, this condition can be approximated by $\ell^{\, 2}\leq \eps\,r^2$. In particular, the triangle refinement is more important for small radii $r>0$, and for sufficiently large $r>0$, no refinement is needed.

For PCA on local neighborhoods, we have two integrals to compute. The first \eqref{eq:mi} corresponds to
\[f(x) = \frac{1}{5}\begin{cases}
( x_ix - r^2e_i)\cdot \nu,\quad\,&\text{if }|x|\leq r\\
0,&\text{otherwise.}
\end{cases}\]
Since $f$ is not Lipschitz, the oscillation bound is at best
\[\osc{T_s} f \leq 2\max_{B_{r+\ell}\setminus B_{r-\ell}}|f| \leq 2r^2, \]
provided $\ell\leq r$. Thus, $\delta$ from \eqref{eq:deltadef} can be chosen as $\delta=2r^2$. Inspecting \eqref{eq:covM2}, we see that it is reasonable to ask that the integration error is bounded by $\eps\,r^4$, for an error tolerance parameter $\eps>0$. Combining this with \eqref{eq:ferror} the restriction on $\ell$ becomes $\ell\leq \min\{\eps\,r,r\}$.

The second integral \eqref{eq:cij} required by PCA on local neighborhoods corresponds to
\[f(x) = \frac{1}{10}\begin{cases}
\bbk{2x_ix_jx - r^2(x_je_i + x_ie_j)}\cdot \nu,\quad\,&\text{if }|x|\leq r\\
0,&\text{otherwise.}
\end{cases}\]
As before, we bound the oscillation by
\[\osc{T_s} f \leq 2\max_{B_{r+\ell}\setminus B_{r-\ell}}|f| \leq 4\,r^3, \]
provided $\ell\leq r$, and so $\delta=4\,r^3$. By \eqref{eq:covM2} we see that it is natural to bound the integration error by $\eps\,r^5$, yielding again the condition $\ell\leq \min\{\eps\,r,r\}$.

\section{Numerical experiments}
\label{sec:numerics}

We now present the results of numerical experiments using our method to compute the spherical volume invariant for triangulated surfaces arising from standard images, and for real experimental data arising from a project to classify and reassemble broken bone fragments in an archaeological context. For brevity, we will not discuss the much simpler case of curves and the circular area invariant. 
Our code is written in C and can be run from Matlab via the MEX interface, and from Python via an extension module. The code is available for download on GitHub:
 \hfill\break \hglue1.27in \href{https://github.com/jwcalder/Spherical-Volume-Invariant}{https://github.com/jwcalder/Spherical-Volume-Invariant}

We first consider the standard test case of the Stanford dragon \cite{curless1996volumetric}. Figure \ref{fig:SVIplotsDragon} shows the spherical volume invariant for radii $r=1,2,5$ computed on the dragon. In Figure \ref{fig:SVIplotsDragon}, and in all other plots below (unless otherwise specified), the colors indicate the values of the spherical volume invariant, with red indicating the lowest value and blue corresponding to the highest.  For the dragon, and all other experiments, we used an error tolerance of $\eps= 1$ for bisecting boundary triangles. In the case of the dragon, the maximum triangle bisection depth was 8 and the maximum number of sub-triangles in any refinement was 57.

\begin{figure}
\centering
\subfloat{\includegraphics[width=0.32\textwidth,clip=true,trim=150 280 150 200]{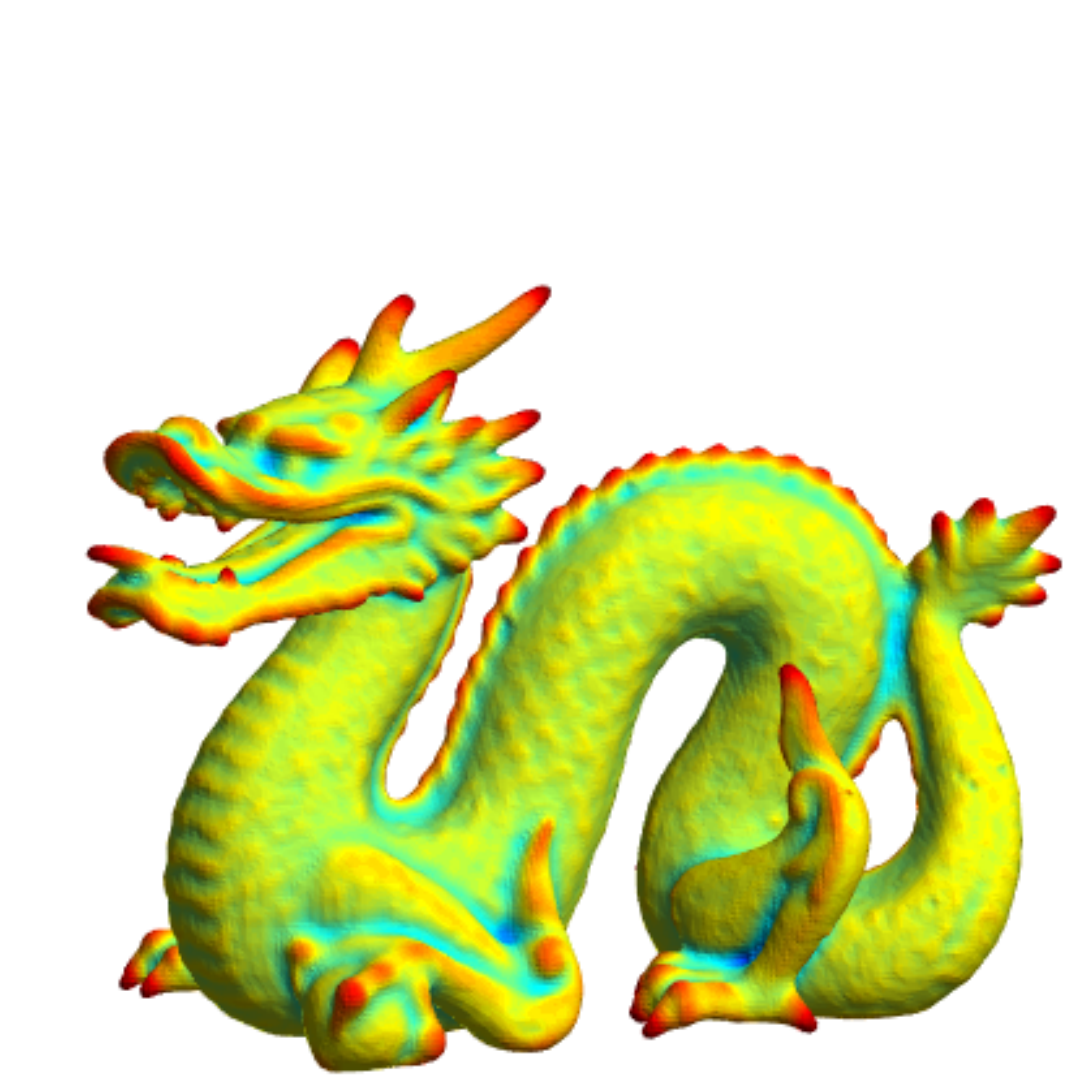}}
\subfloat{\includegraphics[width=0.32\textwidth,clip=true,trim=150 280 150 200]{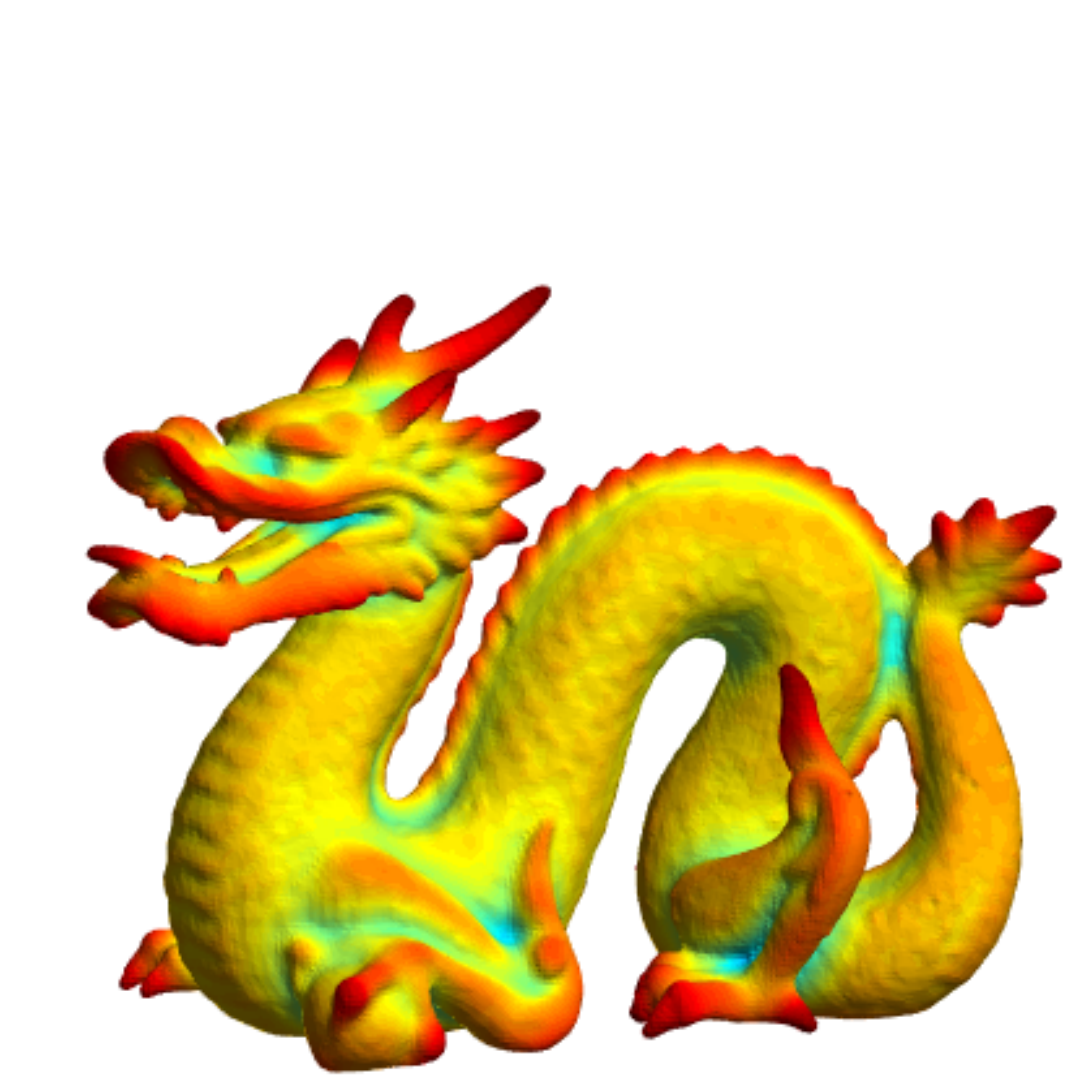}}
\subfloat{\includegraphics[width=0.32\textwidth,clip=true,trim=150 280 150 200]{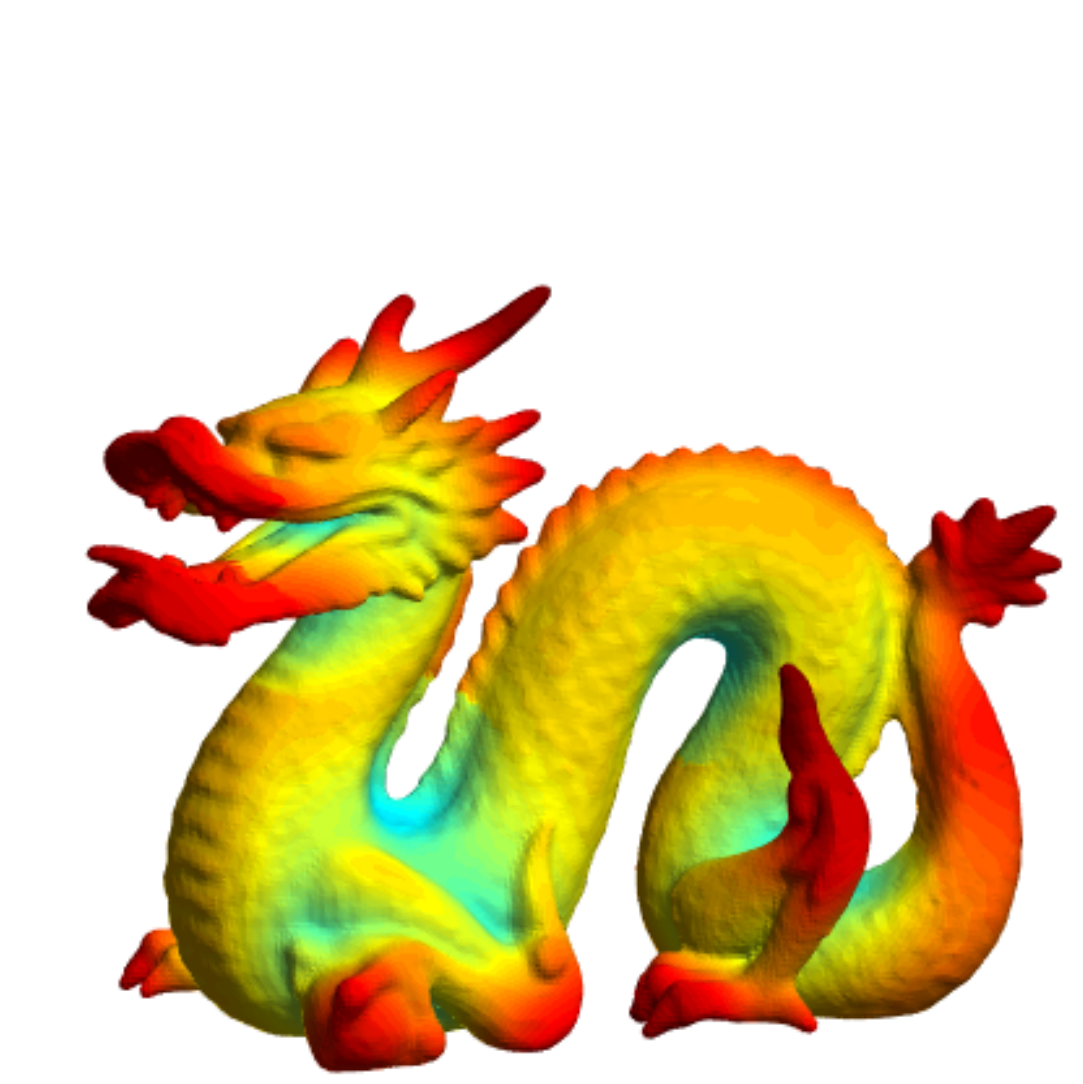}}
\caption{Spherical volume invariant for Stanford dragon \cite{curless1996volumetric} computed at radii of 1, 2, and 5.}
\label{fig:SVIplotsDragon}
\end{figure}

We mention that the original version of the Stanford dragon exhibits some non-manifold geometry. In particular, there are stray vertices not connected to triangles, and some edges are shared by more than 2 triangles. On such meshes, our method can produce unpredictable results, such as negative values for volumes, since Theorem \ref{thm:Vsr} no longer holds. This can be easily remedied by cleaning the mesh with any standard mesh software package before running our code, or obtaining the mesh from a reliable algorithm, such as isosurfacing. For our experiment with the Stanford dragon reported in Figure \ref{fig:SVIplotsDragon}, we obtained a version of the Stanford dragon in PWN (Points with Normals) format from \cite{pwn2ply}, and converted to a clean triangulated mesh using the code provided in \cite{pwn2ply}.

Our method is computationally efficient for large meshes. Table \ref{tab:cputime} shows the wall-clock times\footnote{``Wall-clock'' time refers to the actual amount of time taken to perform the operation, as opposed to CPU time, which is often used to refer to how much time the processor spent on the job.} for computing the spherical volume invariant with our method on the dragon for various radii. We also include results for lower resolution versions of the dragon for comparison. We can see the complexity of our method scales quadratically with the radius $r$, as expected. These CPU times are comparable to the FFT methods reported in \cite{pottmann2007principal}.  We note FFT methods require coarsely discretizing the ambient space, resulting in larger numerical errors. 
\begin{table}[!t]
\centering
\begin{tabular}{|c|c|c|c|c|c|c|}
\hline
\textbf{Mesh size}&\multicolumn{6}{c|}{\textbf{Radius}}\\
\hline
 (\# triangles/\#vertices) &$r=0.5$ &$r=1$ & $r=2$ & $r=3$ & $r=4$  & $r=5$  \\
\hline
45,360/22,678 & 0.19s & 0.69s & 2.5s & 6.1s & 10.3s & 16.8s  \\
\hline
90,722/45,359 & 0.67s & 2.1s & 8.9s & 26.2s & 40.7s  & 66.7s \\
\hline
181,444/90,720 & 2.0s & 7.8s & 32.8s & 83.3s & 151.4s &  268.4s\\
\hline
\end{tabular}
\vspace*{5pt}
\caption{Wall-clock times for computing the spherical volume invariant on the Stanford dragon \cite{curless1996volumetric} with $\eps=1$ for boundary triangle refinement. Computations were performed on a standard laptop computer using a single $3.2$ GHz core, and CPU times were found to be very similar with and without boundary triangle refinement. For reference, on the finest mesh (181,444 triangles) with $r=5$, each ball $B(p,r)$ contains on average 7,384.3 triangles.}
\label{tab:cputime}
\end{table}
 


Let us next compute the spherical volume invariant on broken bone fragments that have been scanned and digitized for anthropological applications, as outlined in the introduction. Figure \ref{fig:SVIplots1} shows the spherical volume invariant plotted over bone fragments at different radii, demonstrating how varying the radius allows one to change the scale of detected features. Here, the values of the spherical volume invariant are normalized with a power-law correction $v\mapsto v^p$, with $p=.5$ unless otherwise stated, to maximize contrast for visualization. We note that for larger radii in Figure \ref{fig:SVIplots1}, the spherical volume appears to be discontinuous at distance $r$ from a sharp fracture edge, which may be a desirable feature, depending on the application. This is due to our use of the \emph{connected component} of $B(p,r)\cap S$ in computations, which fails to explore the opposite side of the fragment if $B(p,r)$ does not intersect the fracture edge. We also computed the principal curvatures via PCA on local neighborhoods. Figure \ref{fig:Gauss} shows the Gauss curvature and Figures \ref{fig:K1} and \ref{fig:K2} show the two principal curvatures for some of the fragments. We did not include figures for mean curvature, since they are identical to those in Figure \ref{fig:SVIplots1} for the spherical volume invariant, except with the colors reversed.

\begin{figure}
\centering
\hspace{-5mm}
\subfloat{\includegraphics[height=0.25\textheight,angle=30,clip=true,trim=220 55 180 40]{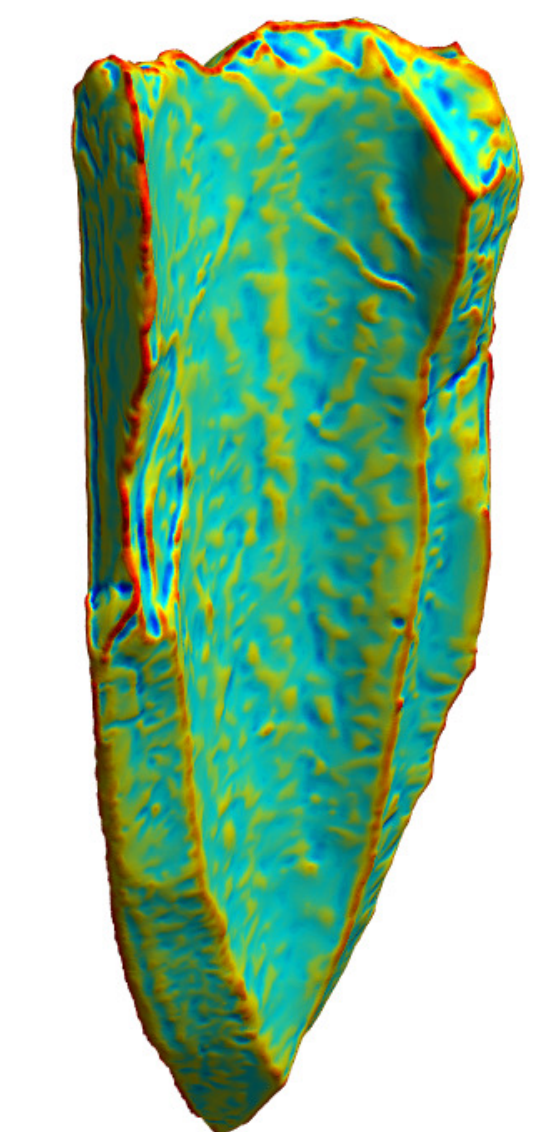}}
\subfloat{\includegraphics[height=0.25\textheight,angle=30,clip=true,trim=220 55 180 40]{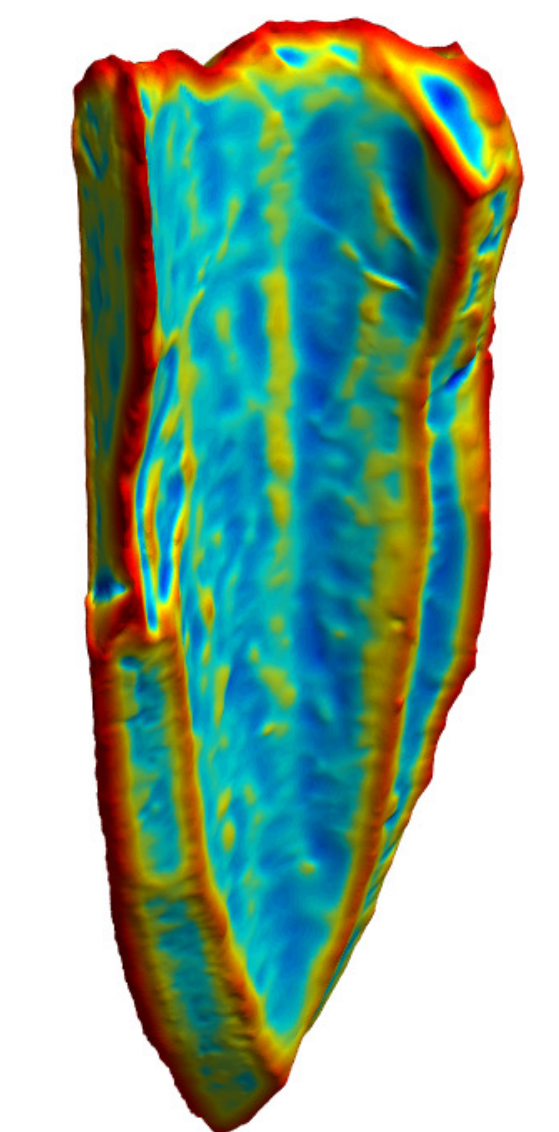}}
\subfloat{\includegraphics[height=0.25\textheight,angle=30,clip=true,trim=220 55 180 40]{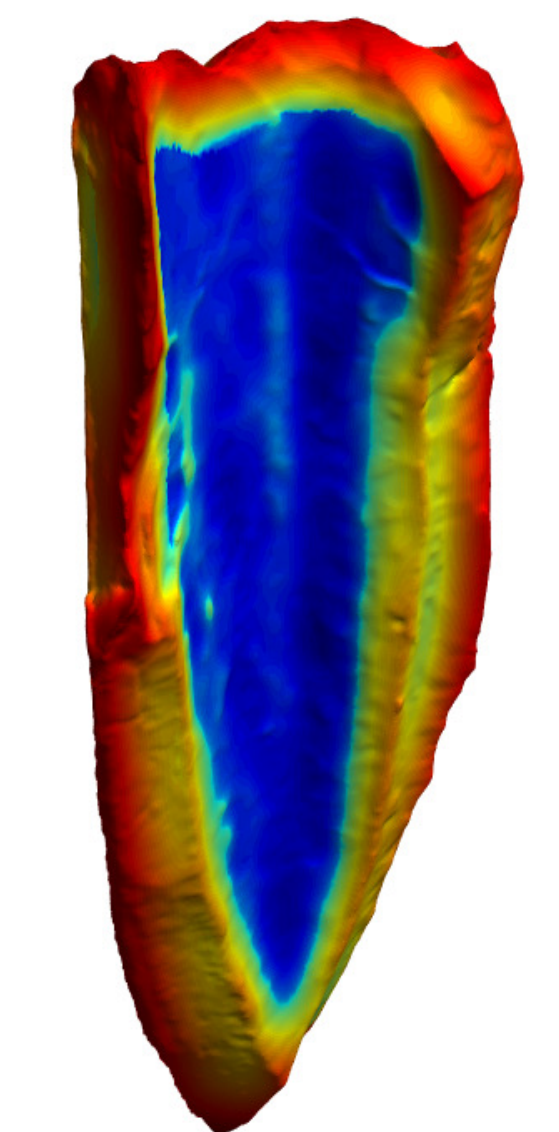}}\\
\hspace{8mm}
\subfloat{\includegraphics[height=0.225\textheight,clip=true,trim=140 50 100 30]{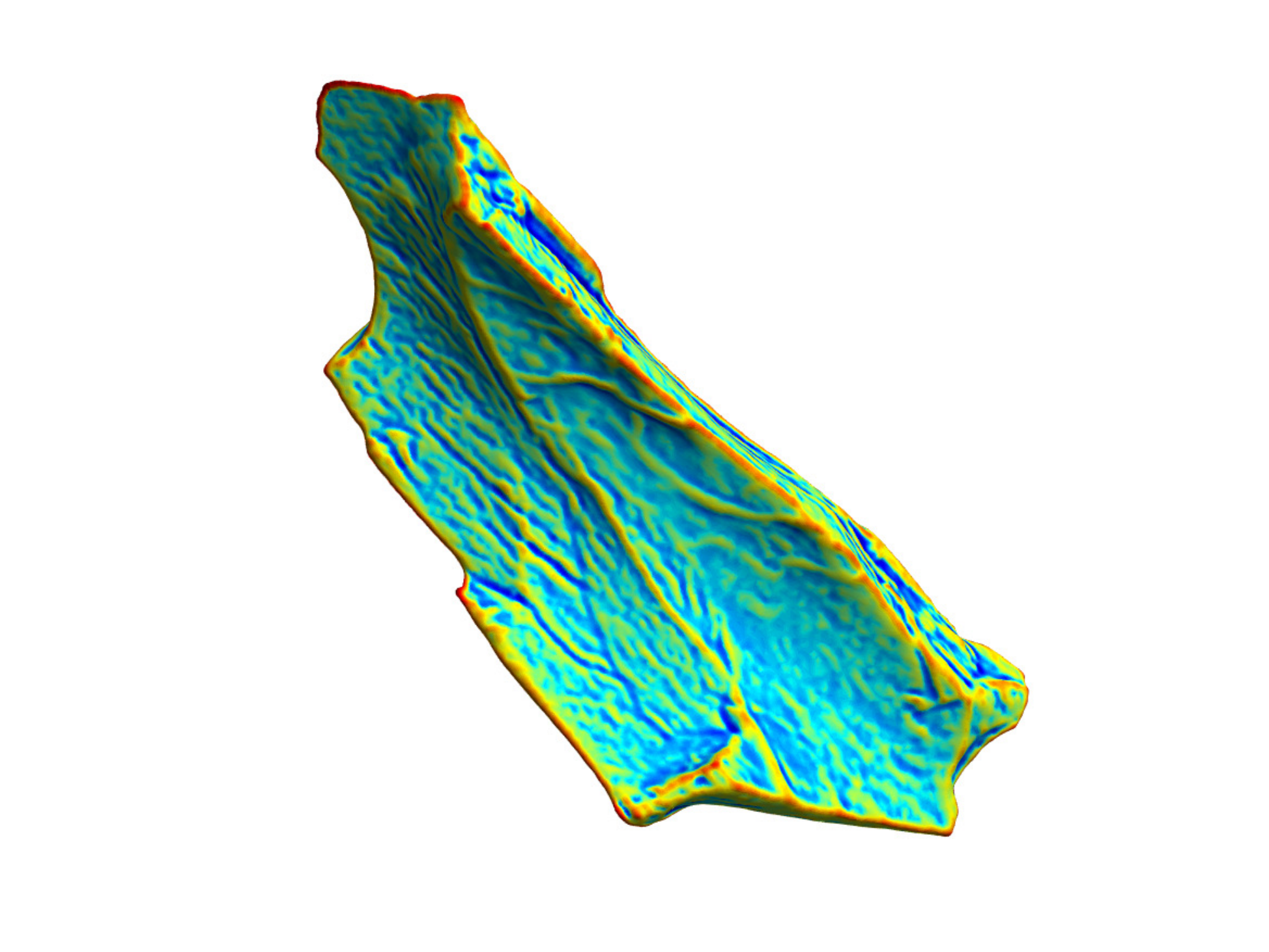}}
\subfloat{\includegraphics[height=0.225\textheight,clip=true,trim=140 50 100 30]{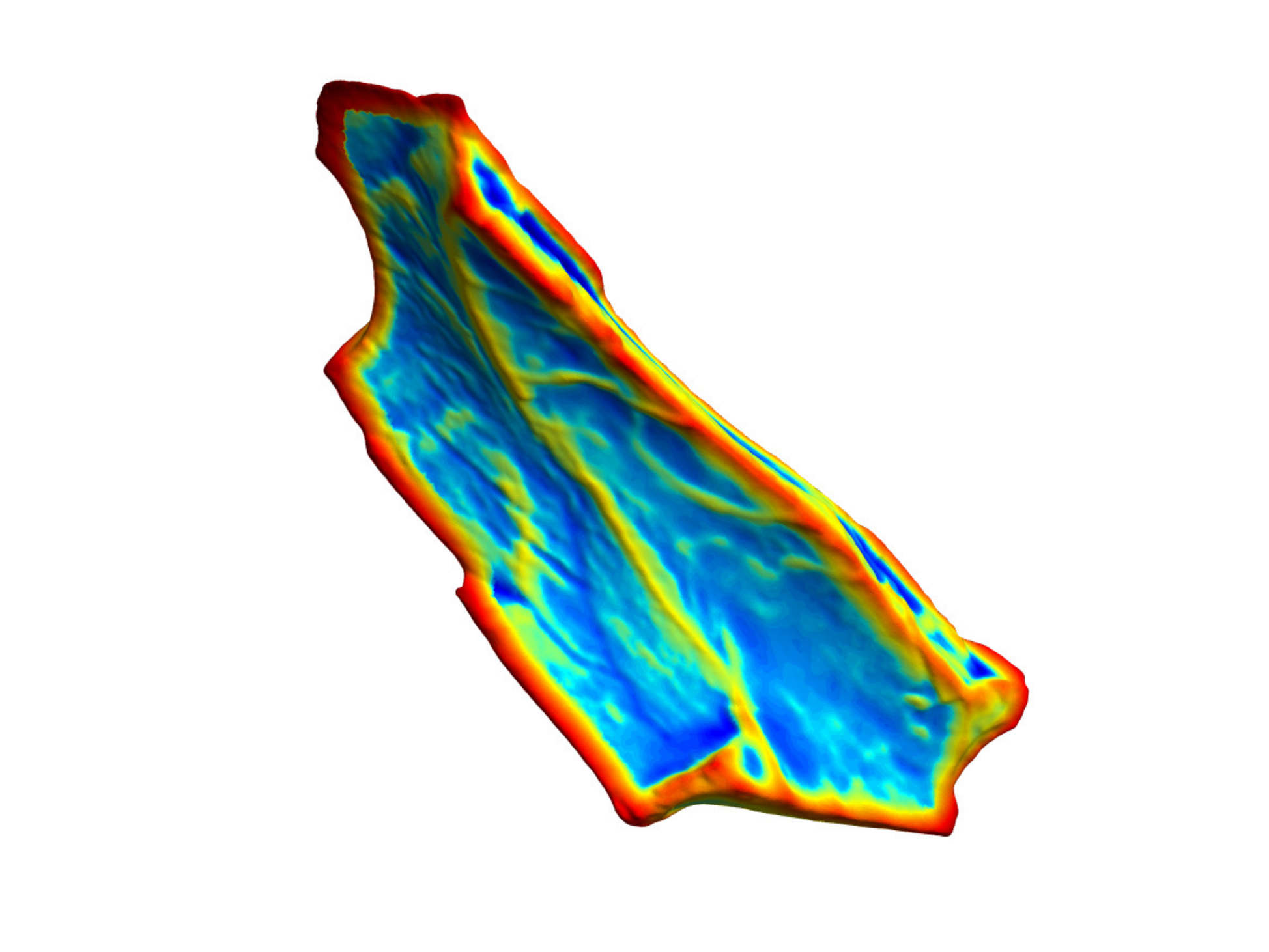}}
\subfloat{\includegraphics[height=0.225\textheight,clip=true,trim=140 50 100 30]{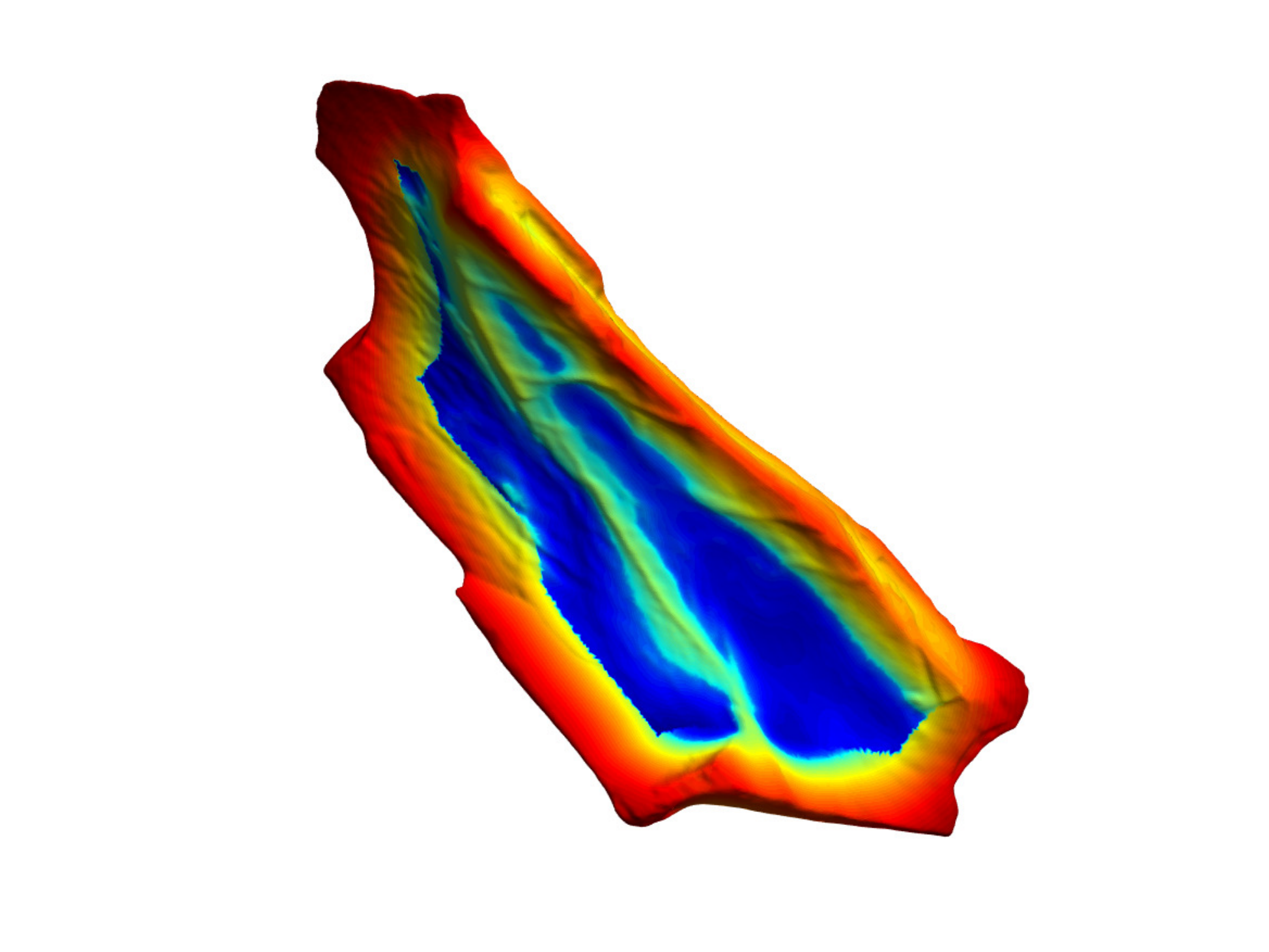}}\\
\vspace{8mm}
\hspace{8mm}
\subfloat{\includegraphics[height=0.2\textheight,clip=true,trim=120 80 140 60]{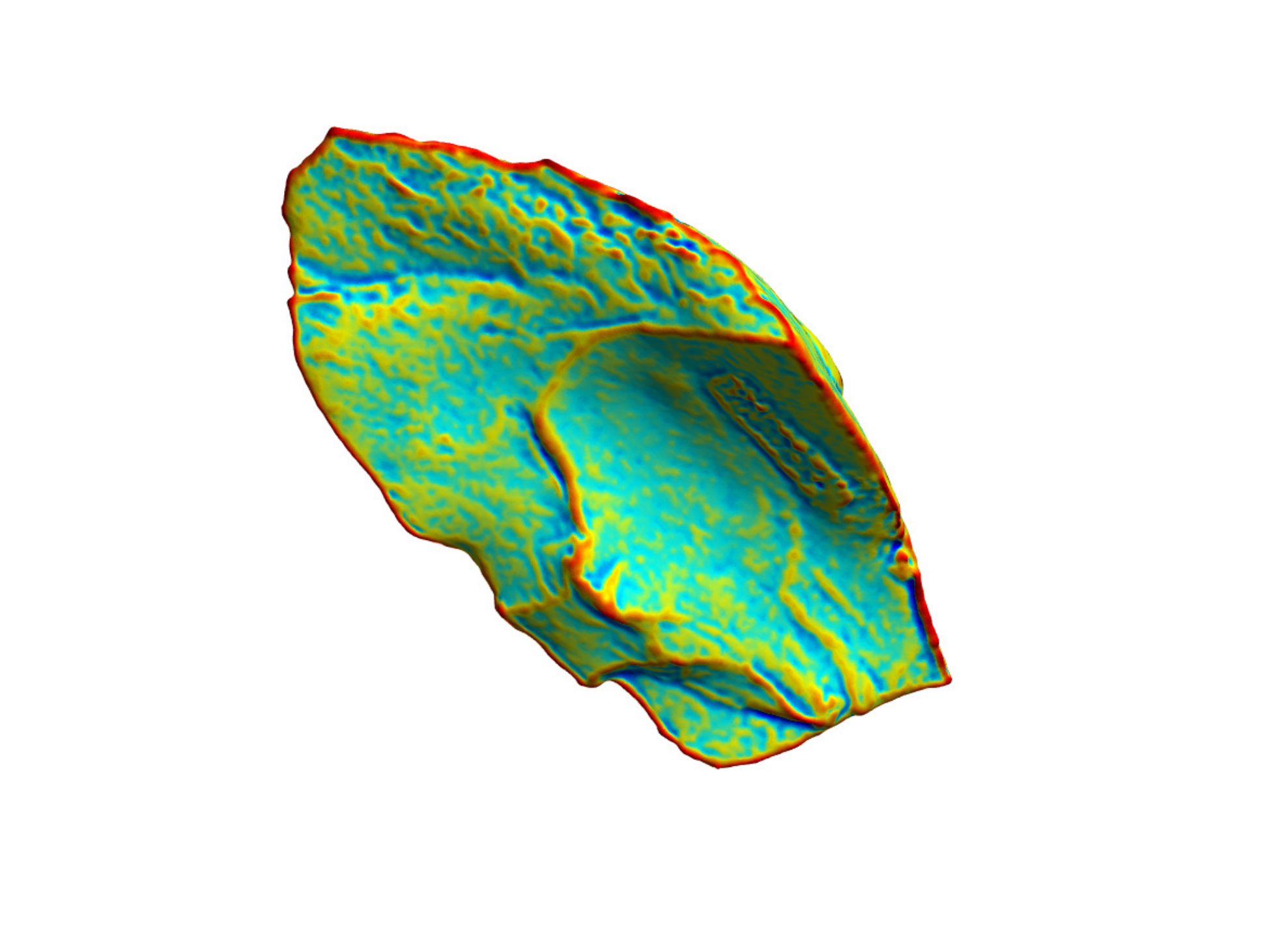}}
\subfloat{\includegraphics[height=0.2\textheight,clip=true,trim=120 80 140 60]{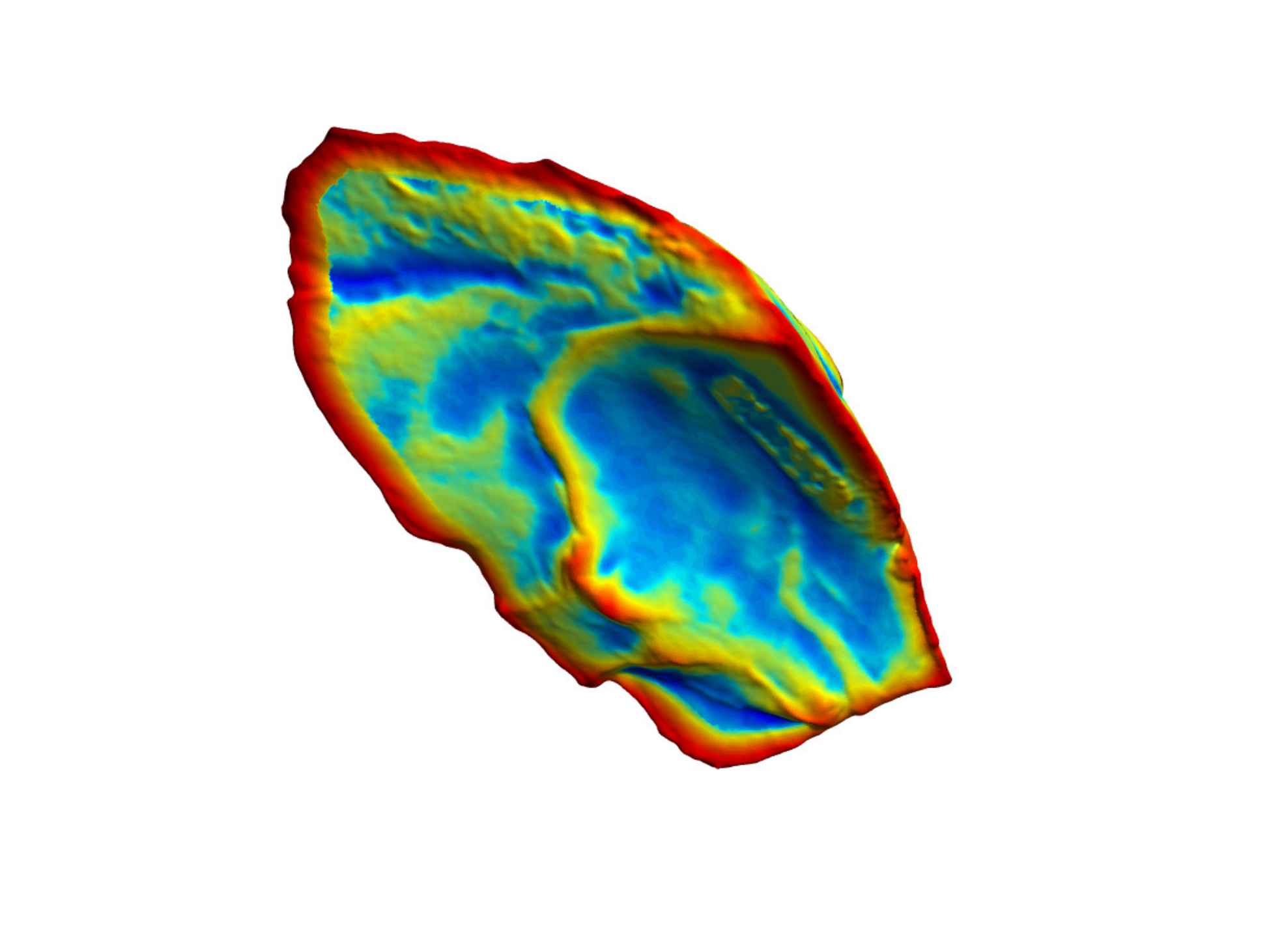}}
\subfloat{\includegraphics[height=0.2\textheight,clip=true,trim=120 80 140 60]{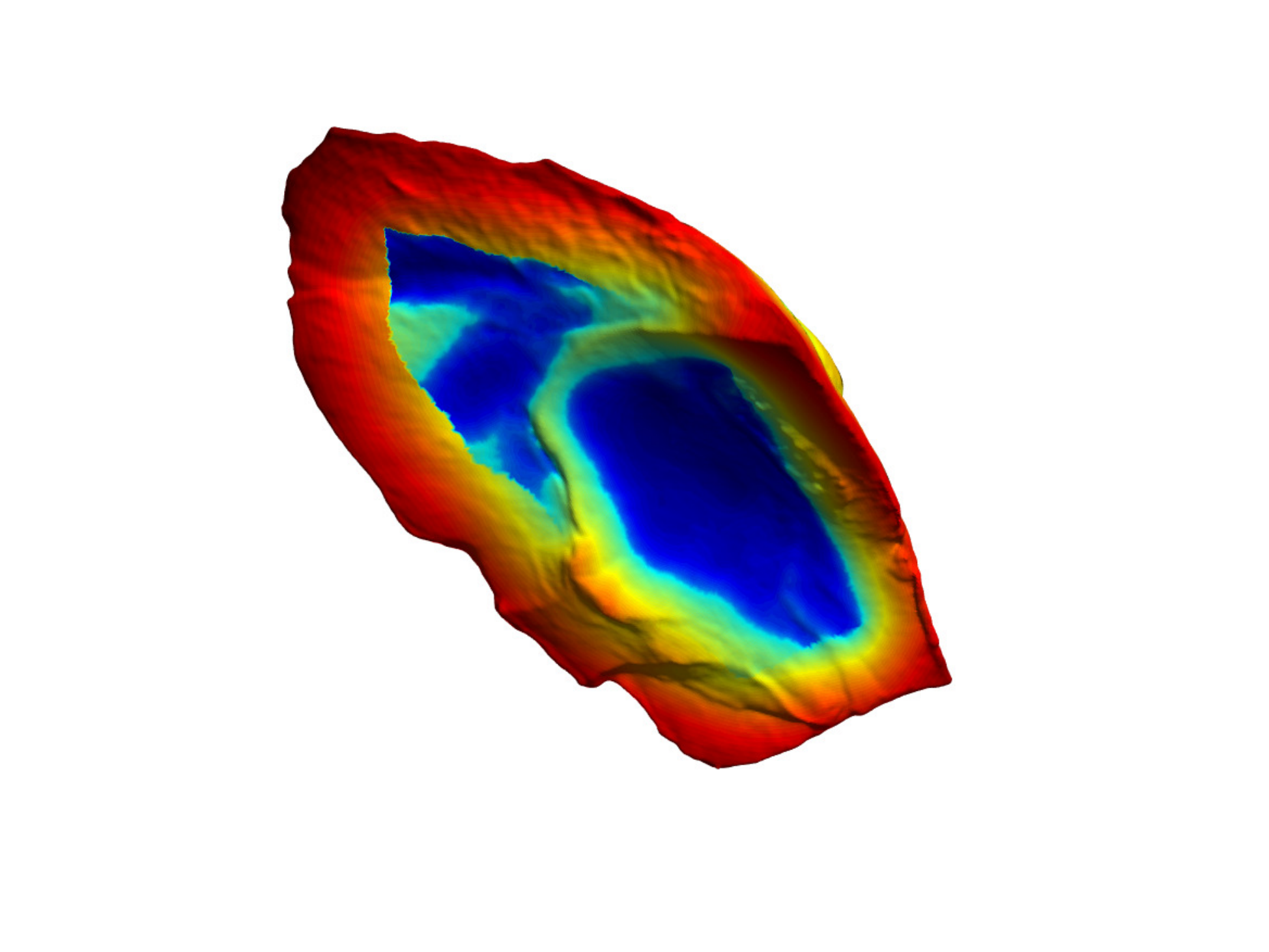}}\\
\caption{Spherical volume invariant computed at radii of 1, 2, and 5.}
\label{fig:SVIplots1}
\end{figure}

\begin{figure}
\centering
\subfloat{\includegraphics[height=0.25\textheight,angle=30,clip=true,trim=220 55 180 40]{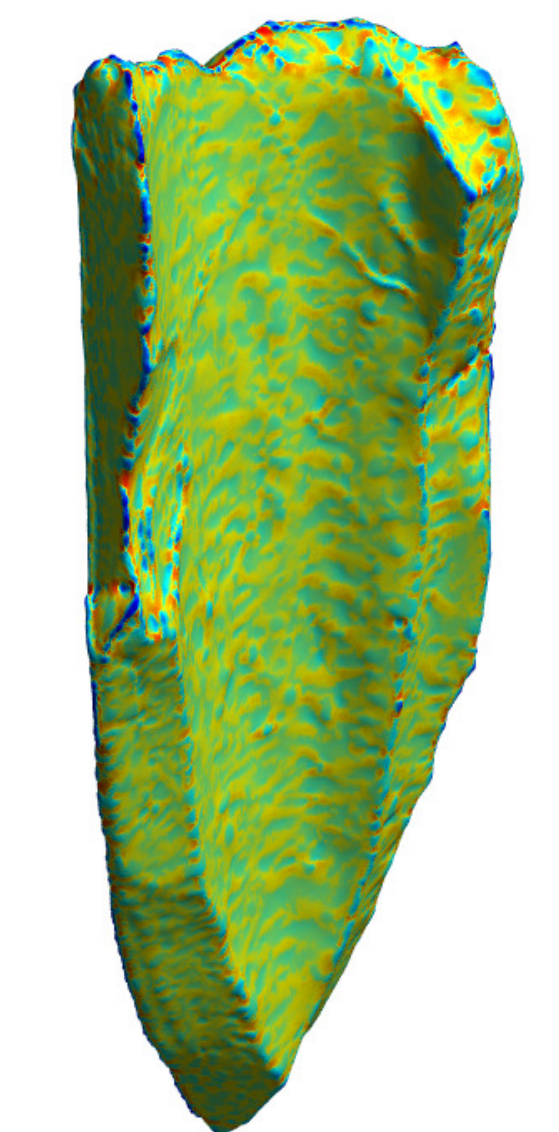}}
\subfloat{\includegraphics[height=0.25\textheight,clip=true,trim=140 50 100 30]{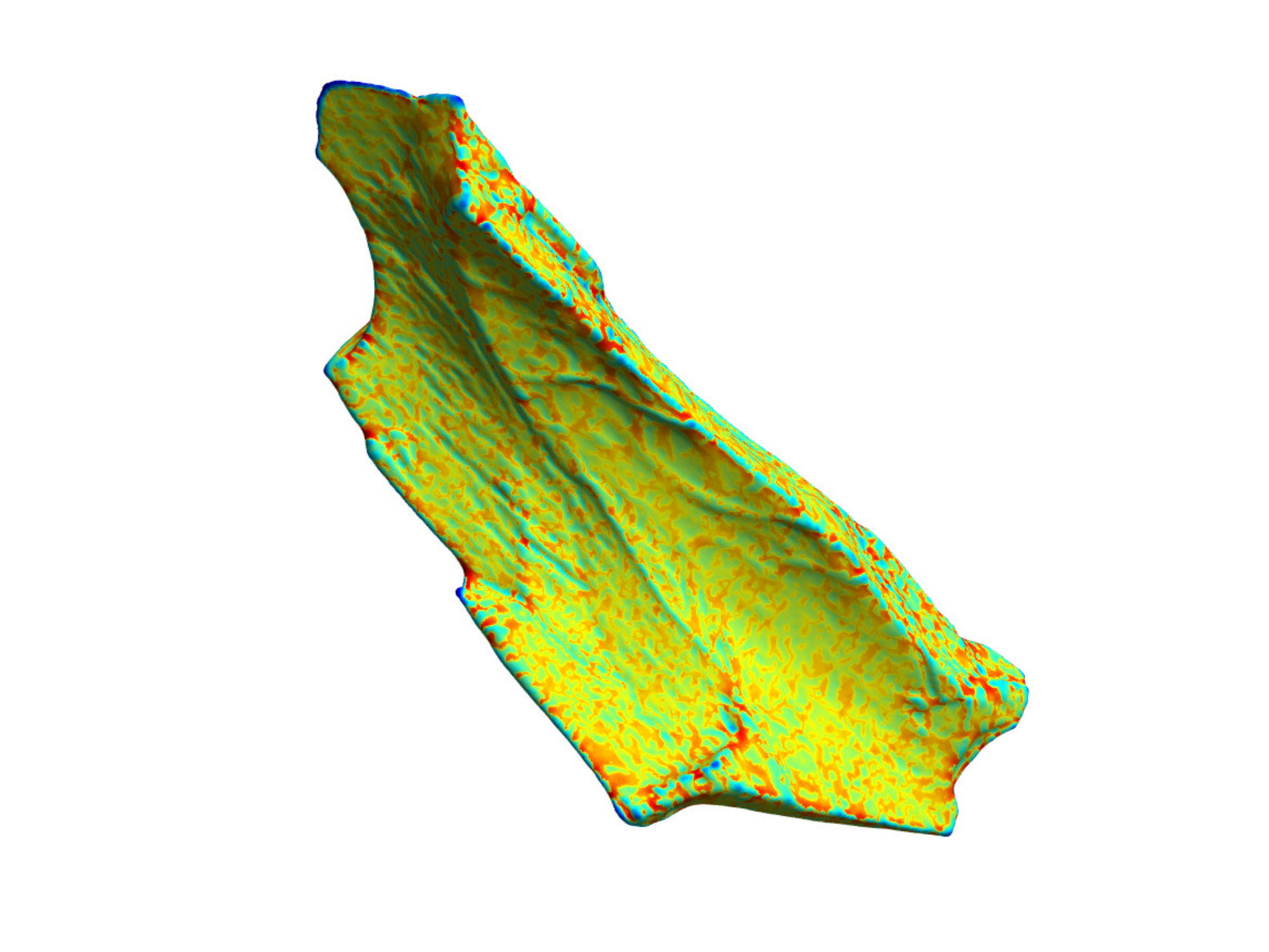}}
\subfloat{\includegraphics[height=0.25\textheight,clip=true,trim=120 80 140 60]{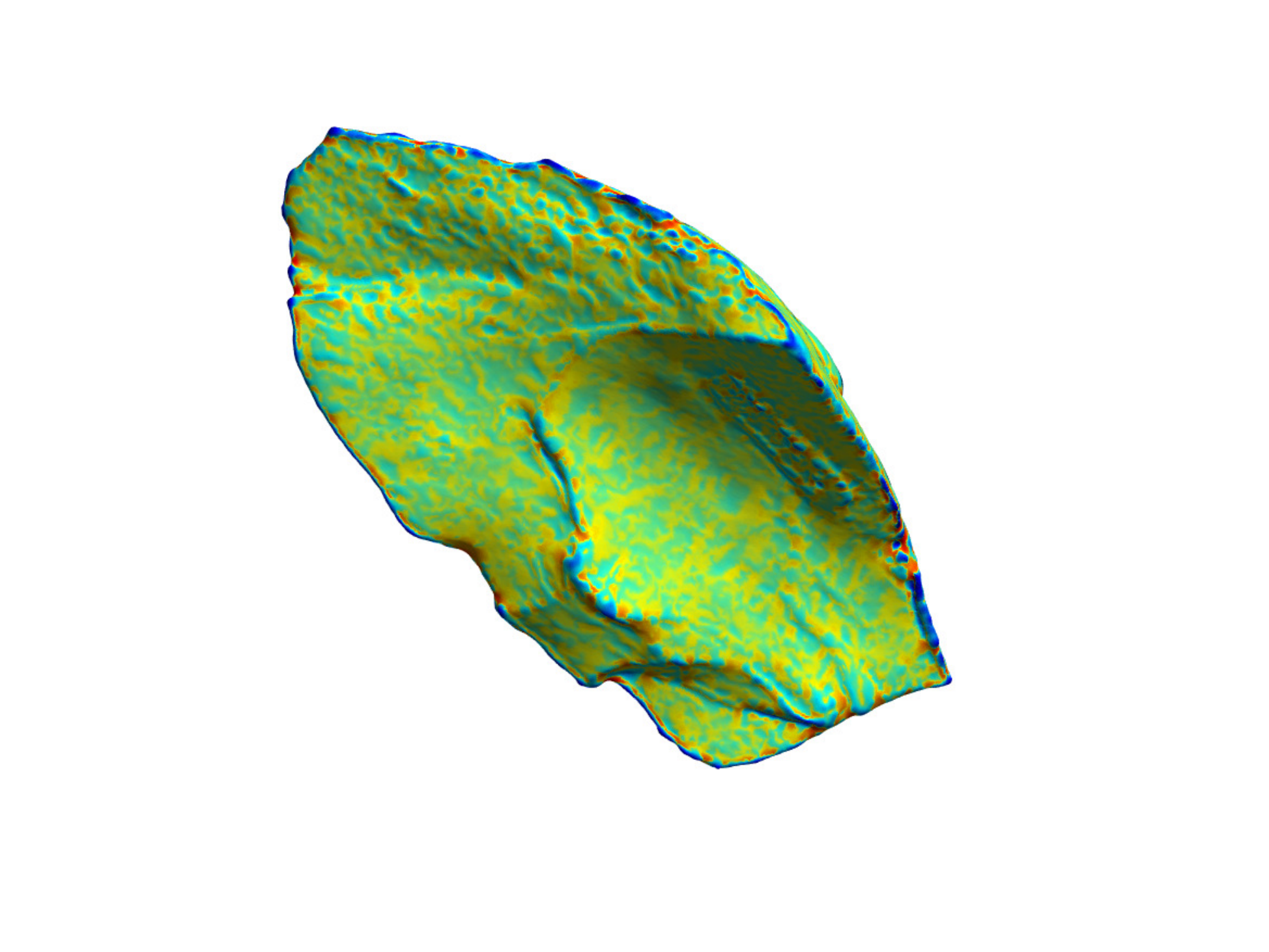}}\\
\caption{Gauss curvature, taken using a radius of 0.5. For power-law correction, $p=0.3$.}
\label{fig:Gauss}
\end{figure}

\begin{figure}
\centering
\subfloat{\includegraphics[height=0.25\textheight,angle=30,clip=true,trim=220 55 180 40]{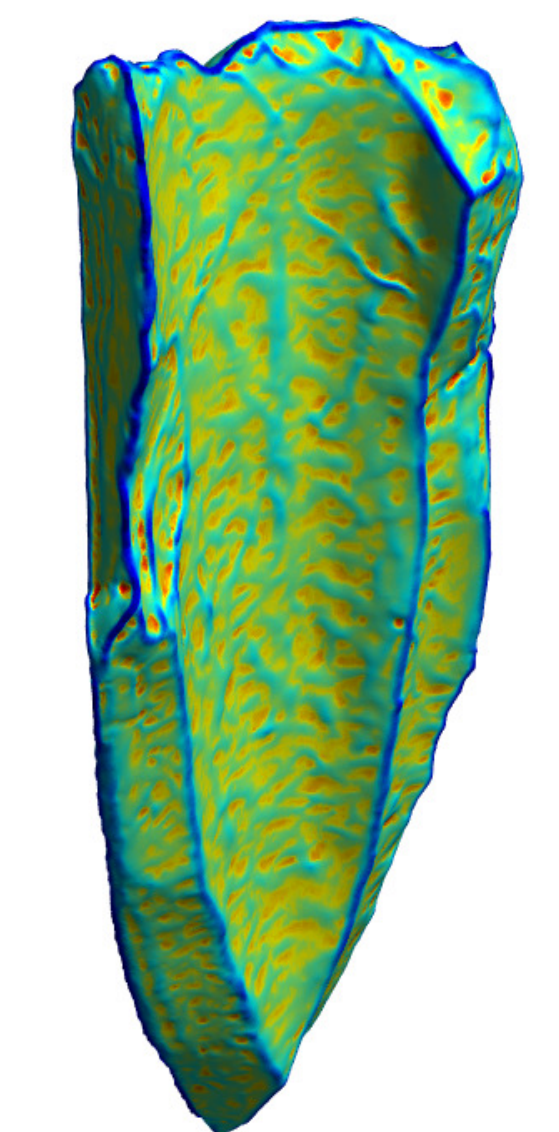}}
\subfloat{\includegraphics[height=0.25\textheight,clip=true,trim=140 50 100 30]{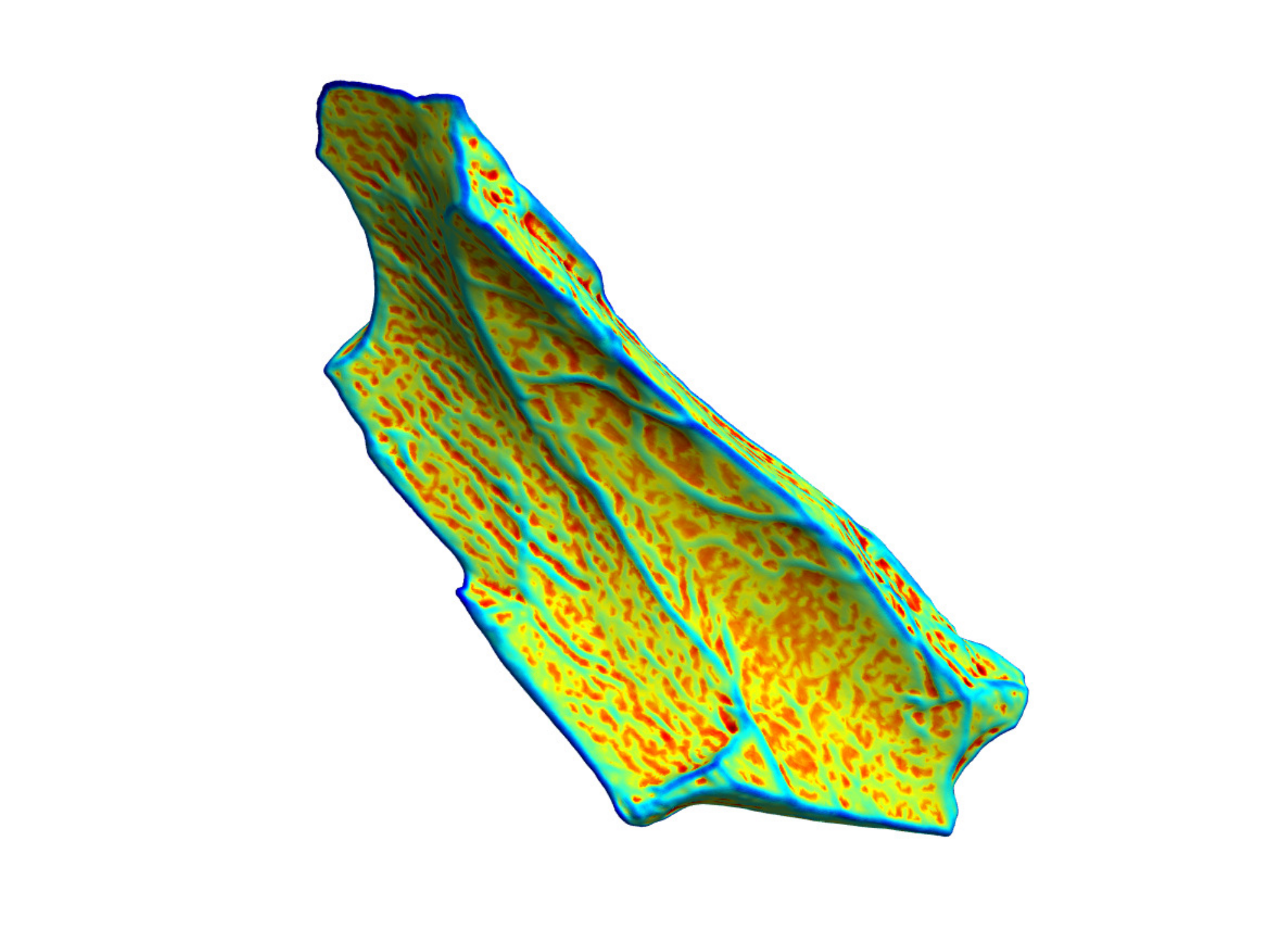}}
\subfloat{\includegraphics[height=0.25\textheight,clip=true,trim=120 80 140 60]{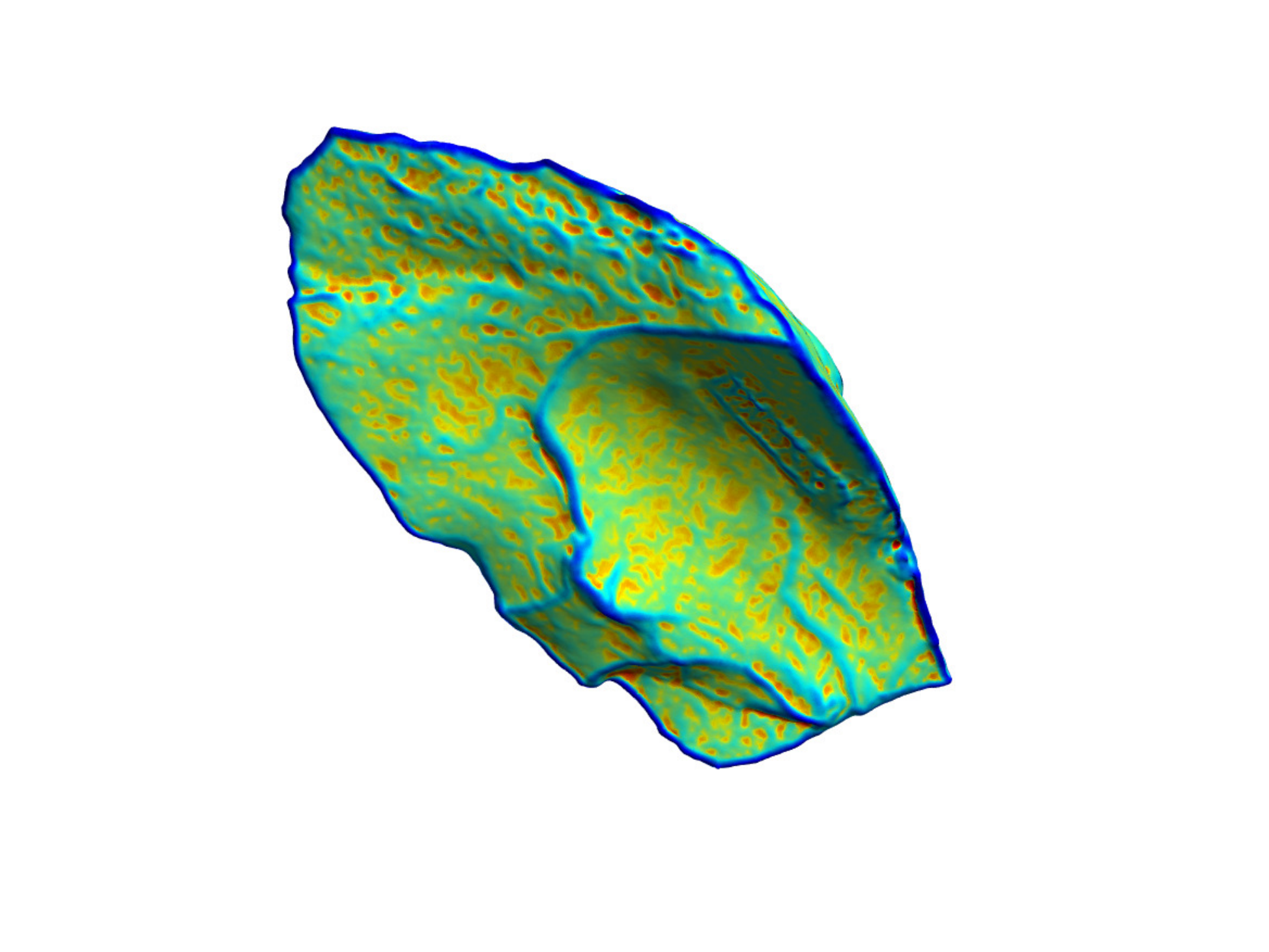}}\\
\caption{Principal curvature $\kappa_1$, taken using a radius of 0.5.}
\label{fig:K1}
\end{figure}

\begin{figure}
\centering
\subfloat{\includegraphics[height=0.25\textheight,angle=30,clip=true,trim=220 55 180 40]{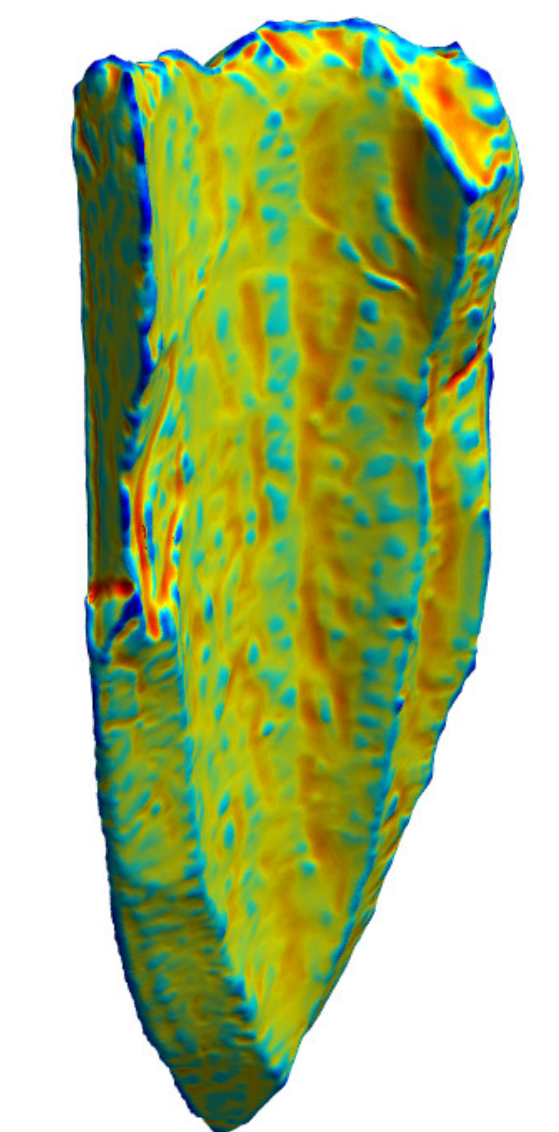}}
\subfloat{\includegraphics[height=0.25\textheight,clip=true,trim=140 50 100 30]{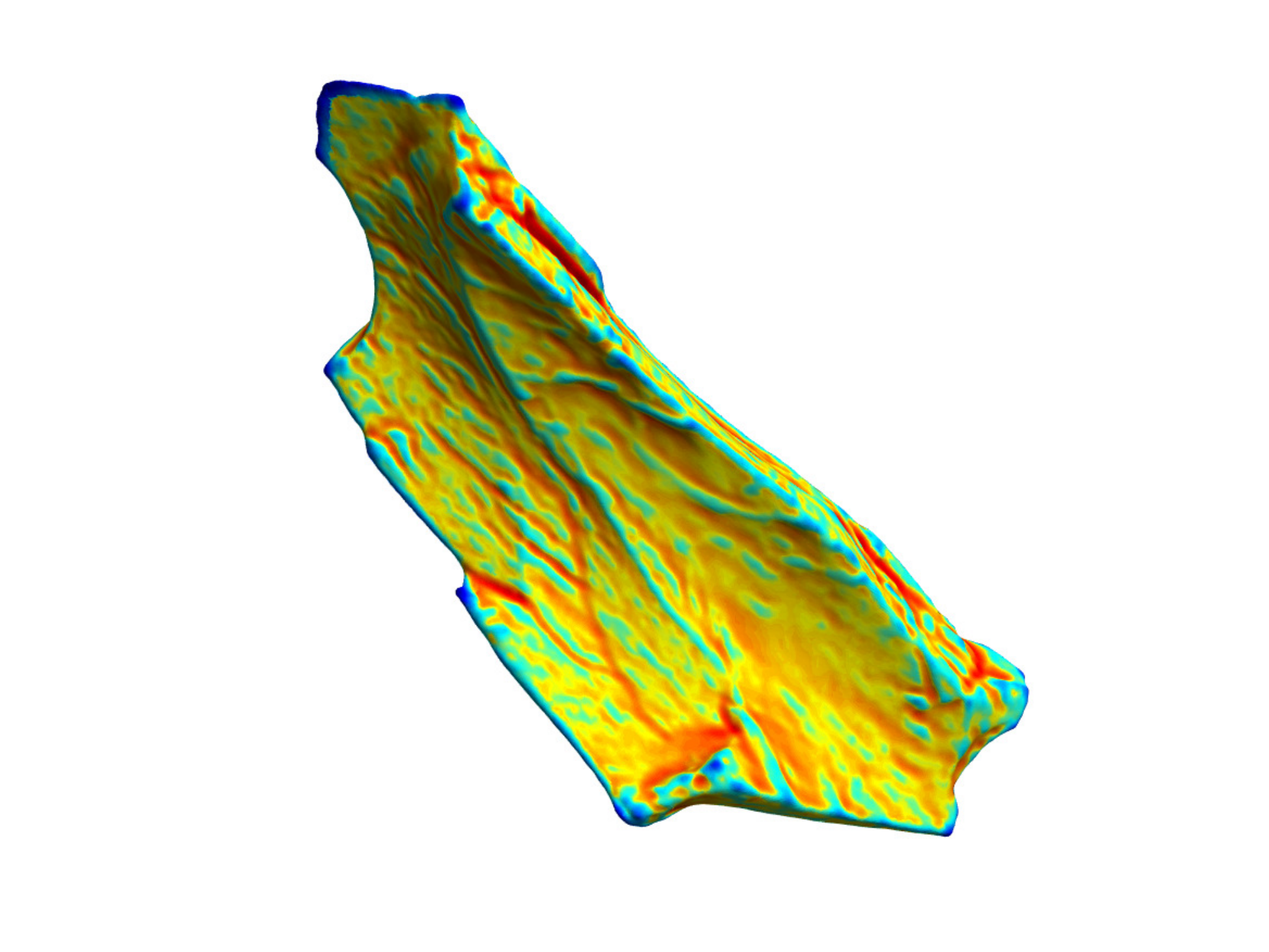}}
\subfloat{\includegraphics[height=0.25\textheight,clip=true,trim=120 80 140 60]{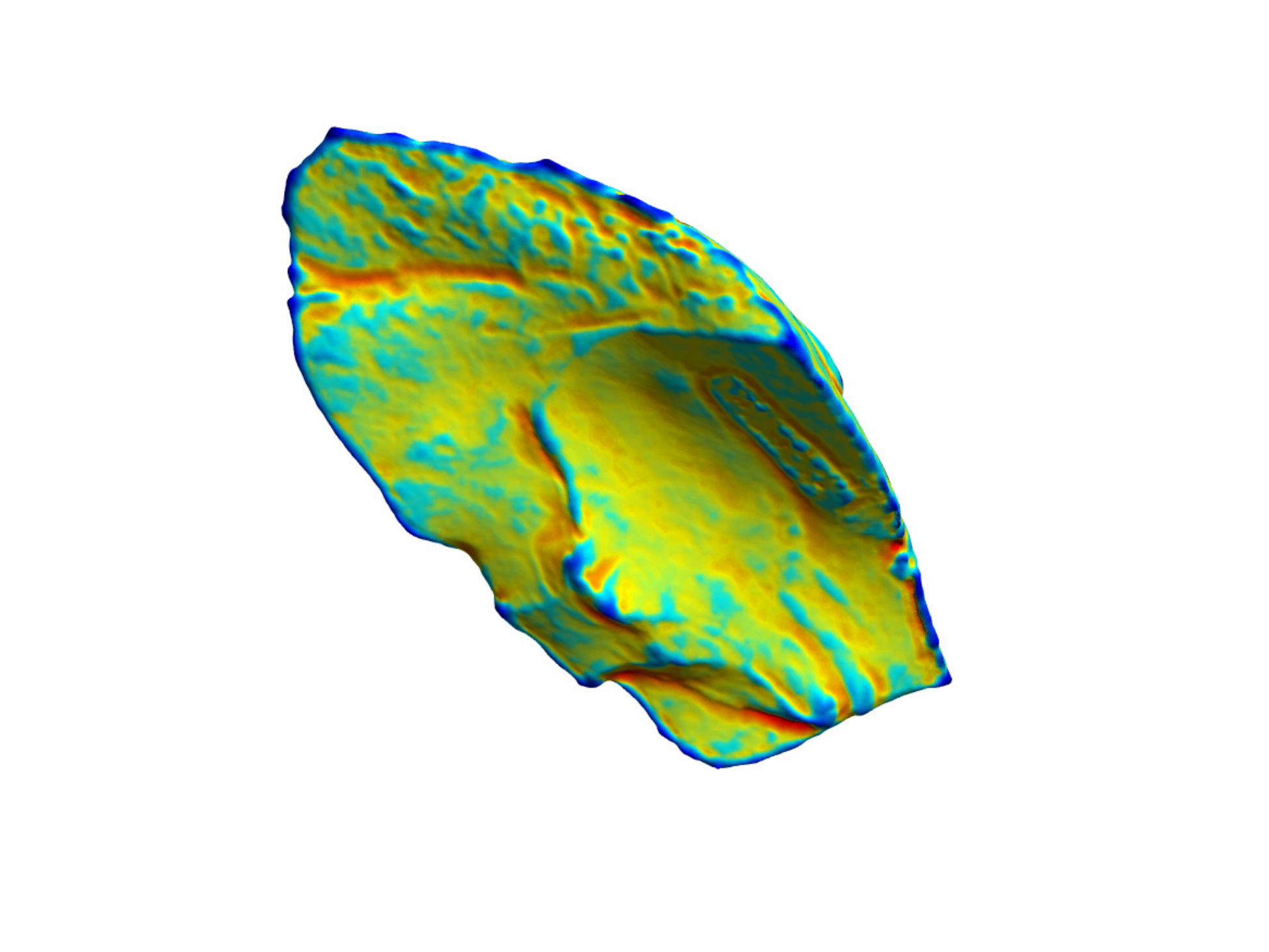}}\\
\caption{Principal curvature $\kappa_2$, taken using a radius of 0.5.}
\label{fig:K2}
\end{figure}


We can detect fracture edges by thresholding the spherical volume invariant. Edge points are taken as those with spherical volumes less than one standard deviation below the mean spherical volumes for the whole fragment. 
Figure \ref{fig:EdgeDetect} shows the results of fracture edge detection on several bone fragments and the Stanford Dragon. This simple approach gives a rough outline of most fracture edges. 
In future work, we plan to investigate automated algorithms for choosing the thresholds as well as the prospect of using the spherical volume invariant with more sophisticated edge detection methods, such as active contours \cite{caselles1997geodesic} on surfaces, or graph-cut segmentation algorithms \cite{merkurjev2013mbo,merkurjev2014graph,ding2001min}.


   
\begin{figure}
\centering


\subfloat{\includegraphics[height=0.25\textheight,angle=30,clip=true,trim=220 55 180 40]{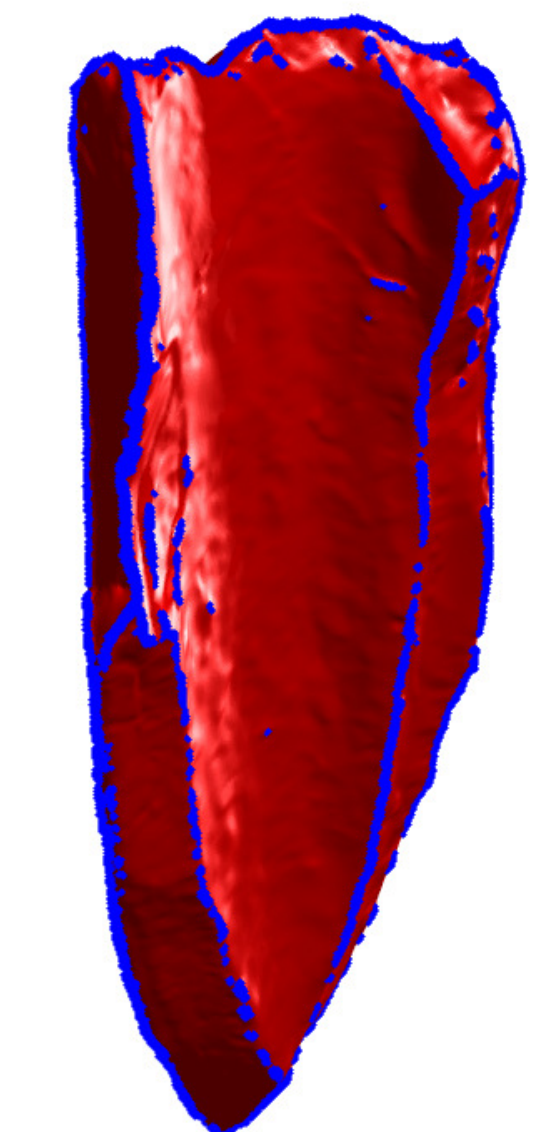}}
\subfloat{\includegraphics[width=0.32\textwidth,clip=true,trim=140 50 140 30]{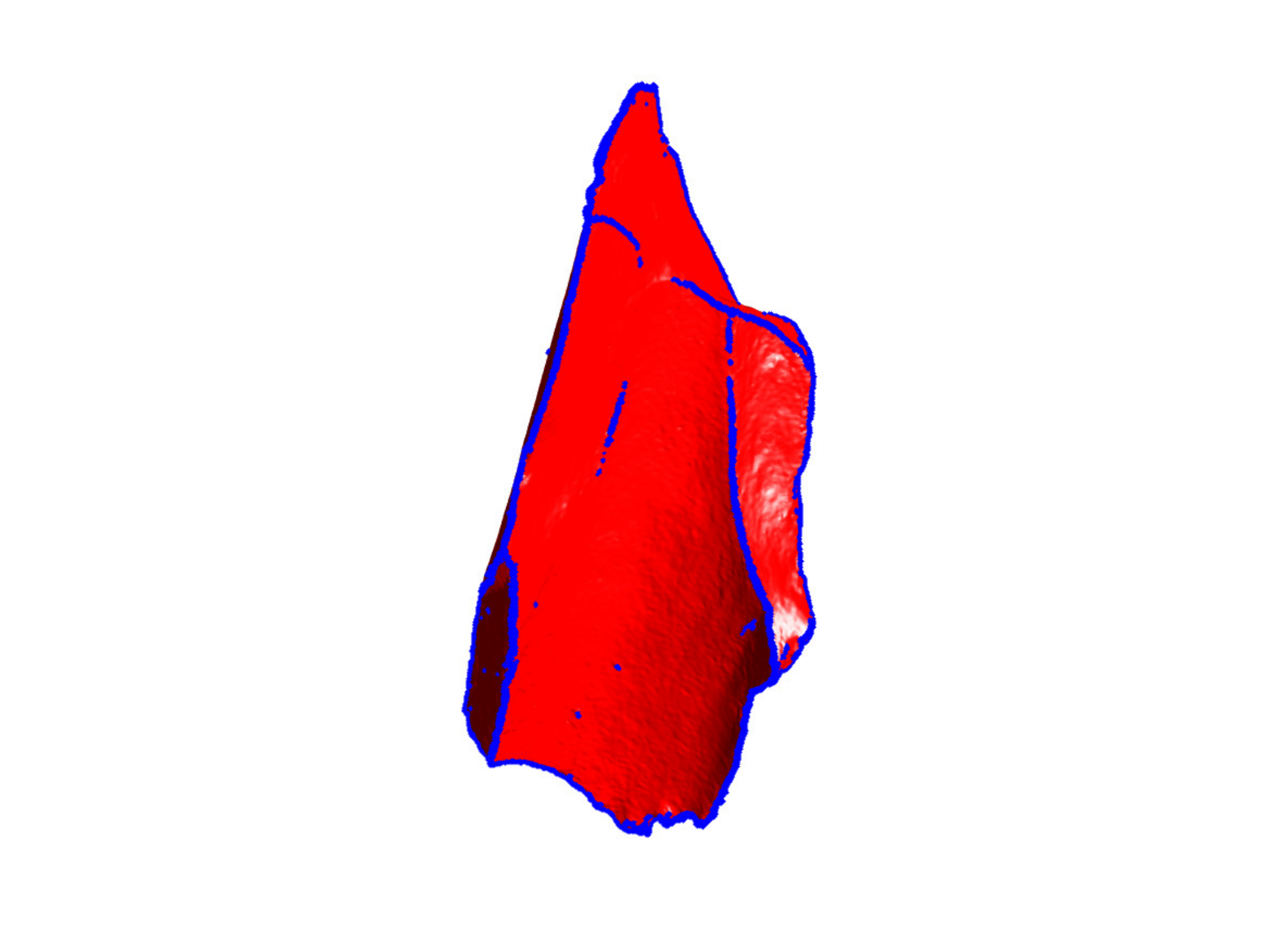}}
\subfloat{\includegraphics[width=0.32\textwidth,clip=true,trim=140 50 100 20]{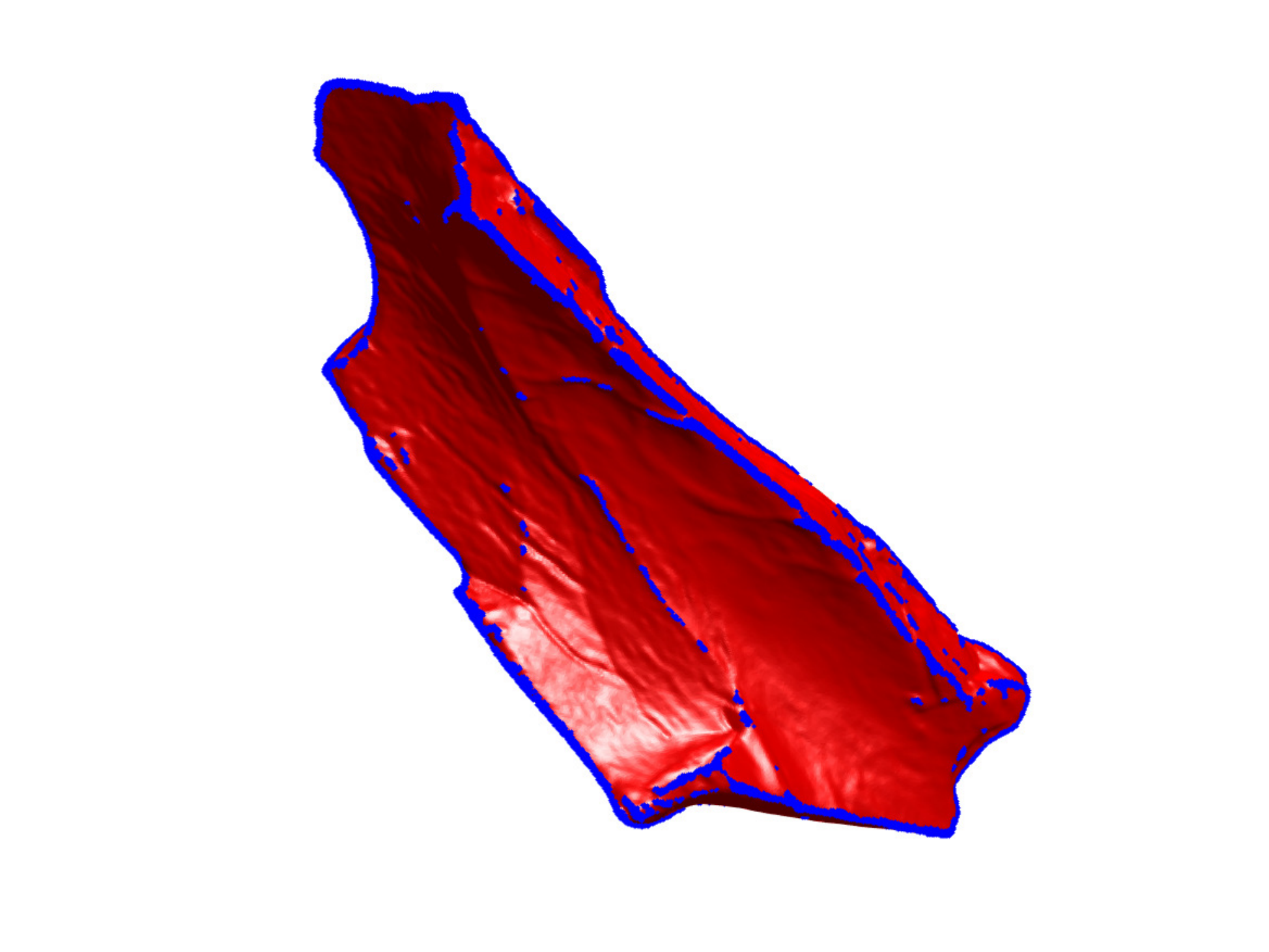}}\\
\subfloat{\includegraphics[width=0.32\textwidth,clip=true,trim=90 40 70 20]{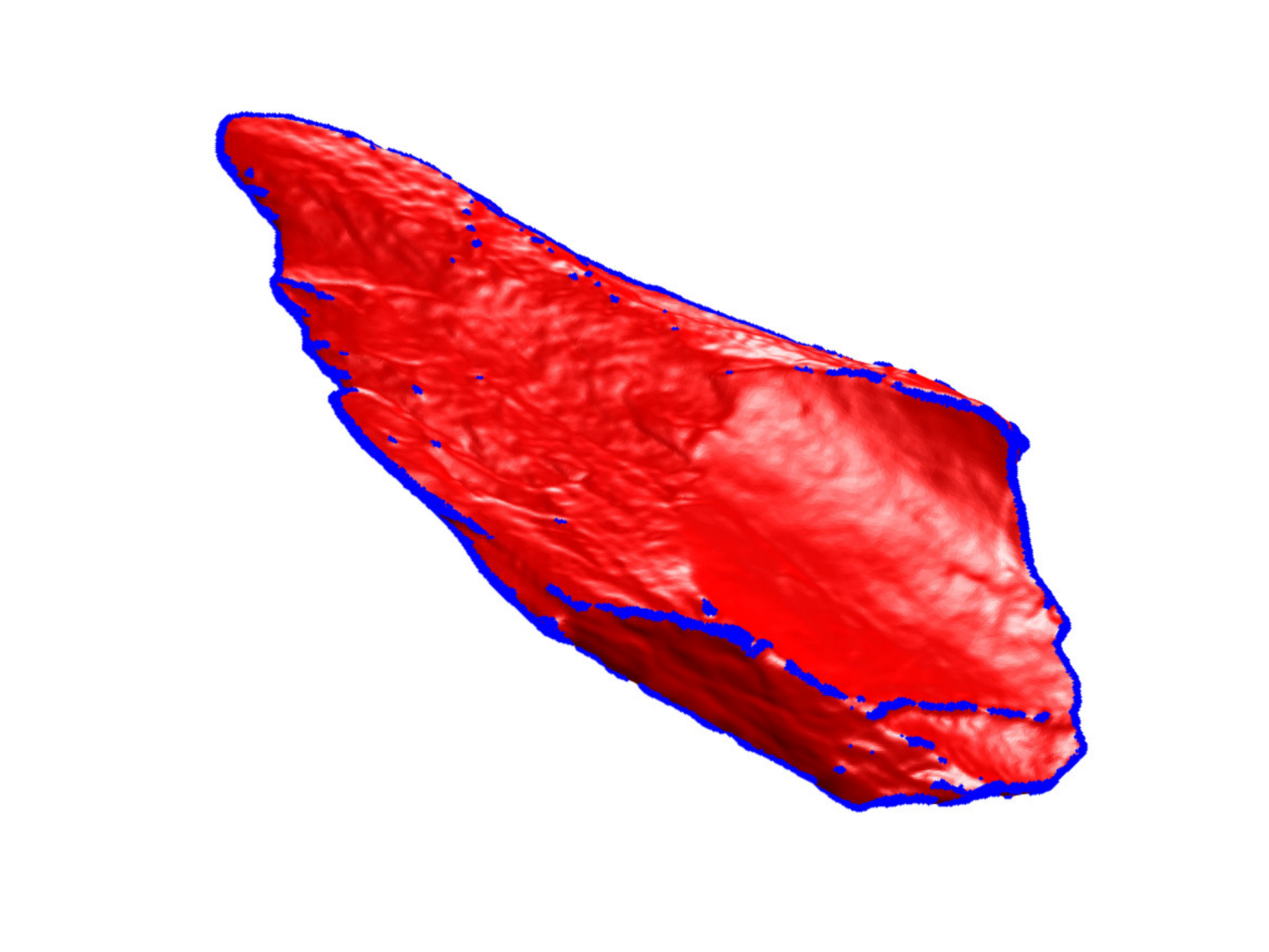}}
\subfloat{\includegraphics[width=0.32\textwidth,clip=true,trim=100 50 90 40]{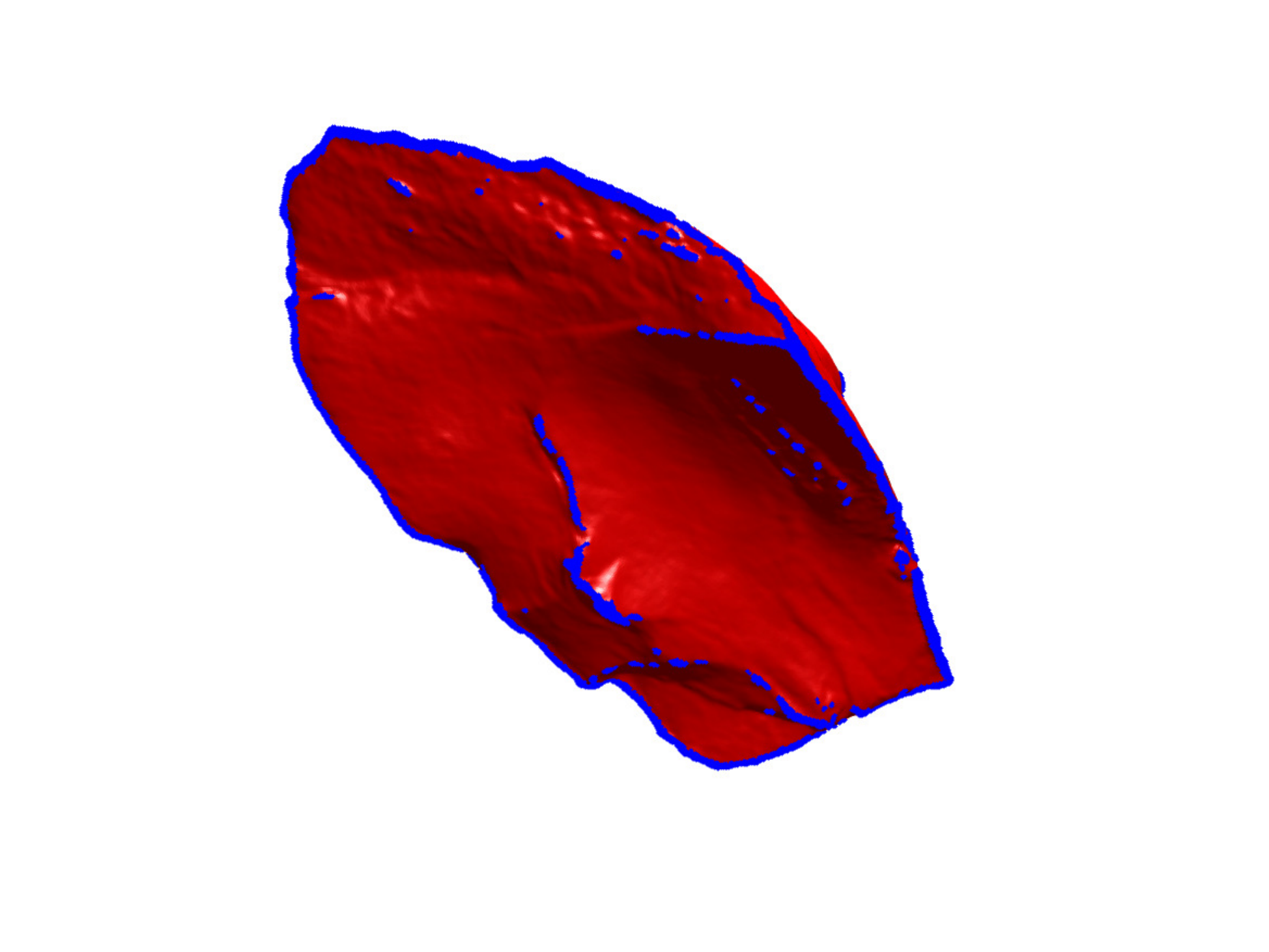}}
\subfloat{\includegraphics[width=0.32\textwidth,clip=true,trim=120 60 95 50]{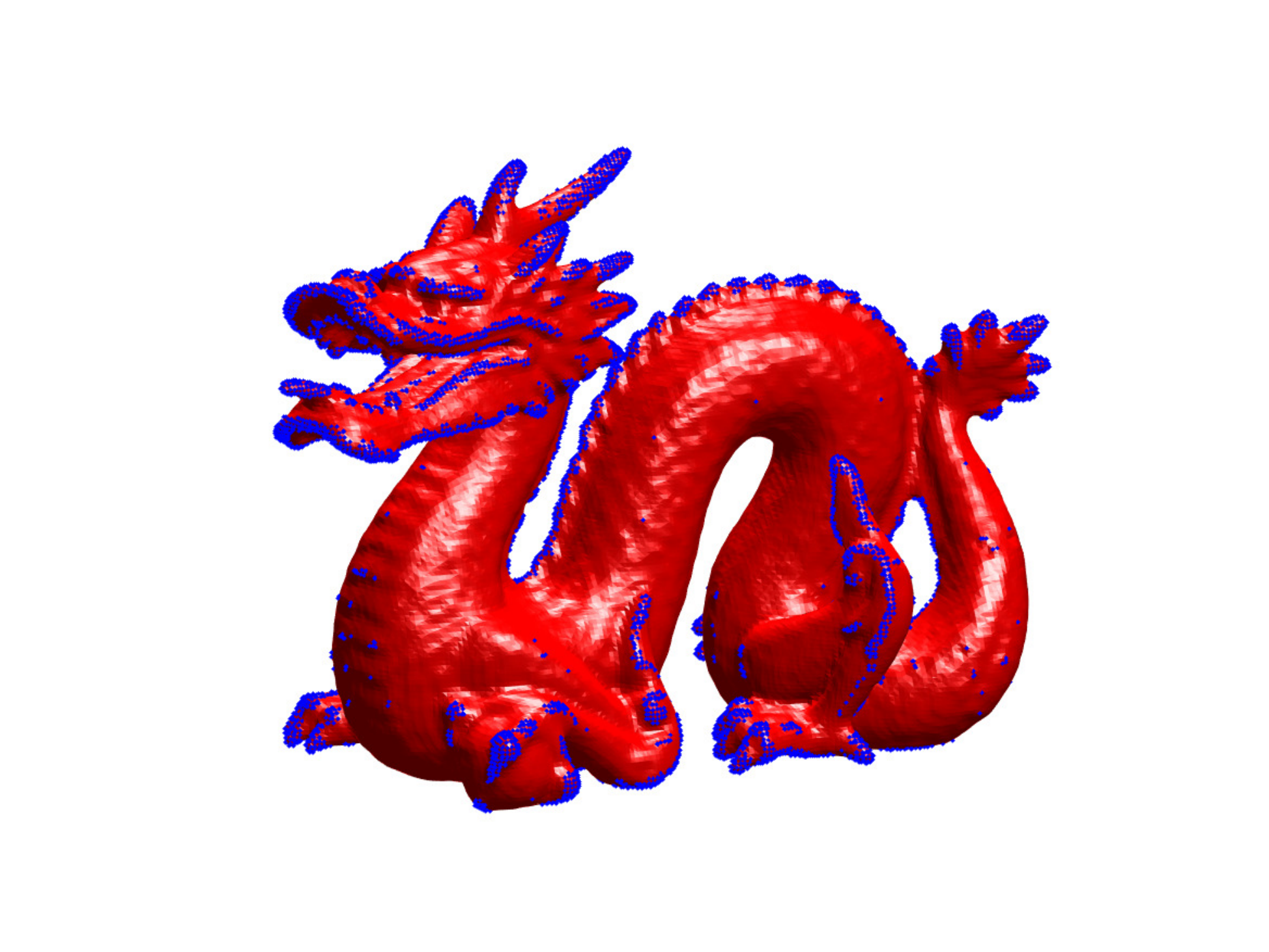}}
\caption{Results of edge detection via thresholding $\kappa_1$ for values at least 1 standard deviation above the mean.}
\label{fig:EdgeDetect}
\end{figure}

\section{Conclusion}
\label{sec:conc}

In this paper, we showed how to compute a class of integral invariants, including the circular area invariant and the spherical volume invariant, in terms of line and surface integrals over the bounding curve and surface. The method is computationally efficient to implement on a triangulated mesh, since it involves simply integrating a function over the mesh triangles, which, when the triangle lies inside the ball, can be done explicitly. In particular, it does not require discretizing the ambient three dimensional space. We showed how to numerically implement the integration accurately and efficiently, and presented the results of some numerical experiments with real data.

\bigskip

{\it Acknowledgements\/}:
 The bone fragments depicted are from an adult elk (\textit{Cervus canadensis}) femur that was broken by adult male spotted hyena (\textit{Crocuta crocuta}) named Scruffy who resides at the Milwaukee County Zoo. The femur was disarticulated and defleshed prior to being fed to the hyena. All fragments were scanned with a DAVID white light scanner that was made available by the Evolutionary Anthropology Labs at the University of Minnesota. The authors also gratefully acknowledge discussions with  Martha Tappen, Jacob Elafandi, and Jacob Theis.

\appendix

\section{Analytic formula for a hypersingular integral}
\label{sec:hyper}

Here, for the reader's convenience, we recall the analytic formula from \cite{nintcheu2009explicit} for the hypersingular integral
\begin{equation}\label{eq:hyperI}
\int_T \frac{1}{|x|^3}\, dx,
\end{equation}
where $T$ is a planar triangle in $\R^3$, such that $0\not\in T$.  (We note that \cite{nintcheu2009explicit} includes analytic formulas for several such triangular hypersingular integrals involving other negative integer powers of $|x|$.)  Let $P$ denote the plane containing $T$ and $\nu$ the unit outward normal vector to $T$ and $P$. In what follows, we take the triangle $T$ to be an open subset of $P$, i.e., $T \cap \partial T = \emptyset$.

Let $x^*\in \R^3$ denote the orthogonal projection of the origin onto the plane $P$.
Let $x^1,x^2,x^3\in \R^3$ be the vertices of $T$, given with positive orientation; for convenience of notation we write $x^4 = x^1$.   Define
\begin{equation}\label{eq:thetaform}
\theta = \begin{cases}\
0,&\text{if }\ \ x^* \in P\setminus \overline{T}\\
\ \pi,&\text{if }\ \ x^* \in \partial T\setminus \{x^1,x^2,x^3\}\\
\ 2\pi,&\text{if }\ \ x^* \in T\\
\ \theta_i,&\text{if }\ \ x^*=x^i,
\end{cases}
\end{equation}
where  $\theta_i$ is the interior angle of $T$ at the vertex $x^i$.

Let $L^i$ denote the oriented edge of the triangle $T$ from $x^i$ to $x^{i+1}$. Associated with each edge $L^i$, we construct an orthonormal basis $(\e^i_1,\e^i_2)$ for the plane $P$ with origin $x^*$, $\e^i_1$ taken in the direction of the edge $L^i$, and $\e^i_2 = \nu\times \e^i_1$ chosen so that $(\e^i_1,\e^i_2,\nu)$ is an orthonormal basis for $\R^3$.
Let 
\[p^j_i = (x^j-x^*)\cdot \e^i_1, \qquad  q^j_i = (x^j-x^*)\cdot \e^i_2,\] 
be the planar coordinates of the vertex $x^j$ in the basis $(\e^i_1,\e^i_2)$.  By definition, $q^1_1=q^1_2$, $q^2_2=q^3_2$, and $q^3_3=q^4_3$, since the vertices $x^j$ and $x^{j+1}$ lie along the line spanned by $\e^j_1$. We denote the common values as
\[q_i :=q^i_i=q^{i+1}_i.\]

Finally, set $\eta = x^1\cdot \nu$, noting that $\eta\neq 0$, since $0\not\in T$.  We then define
\begin{equation}\label{eq:gami}
\gamma_i = \arctan\left( \frac{-2p^i_iq_i\eta|x^i|}{(q_i)^2|x^i|^2 - (p^i_i)^2\eta^2} \right) - \arctan\left( \frac{-2p^{i+1}_iq_i\eta|x^{i+1}|}{(q_i)^2|x^{i+1}|^2 - (p^{i+1}_i)^2\eta^2} \right),
\end{equation}
using the branch of $\arctan$ with values in $(-\pi/2,\pi/2)$. Finally, the hypersingular integral \eqref{eq:hyperI} is given by
\begin{equation}\label{eq:hyperform}
\int_T \frac{1}{|x|^3}\, dx= \frac{\gamma_1 + \gamma_2 + \gamma_3 + 2\,\text{sign}(\eta)\,\theta}{2\eta}.
\end{equation}

\vskip20pt


\end{document}